\documentclass{article}

\usepackage{amsmath}
\usepackage{amsthm}
\usepackage{amssymb}
\usepackage{amsfonts}

\usepackage{subeqnarray}
\usepackage{bm}
\usepackage{tikz}
\usepackage{dsfont}
\usepackage{upgreek}
\usepackage{booktabs}
\usepackage{fancyhdr}
\usepackage{multirow}
\usepackage{multicol}
\usepackage{float}
\usepackage[inline]{enumitem}   
\usepackage{cases}

\usepackage[ruled,vlined,linesnumbered,inoutnumbered]{algorithm2e}

\usepackage{bbding}
\usepackage{caption}
\usepackage{todonotes}

\usepackage{mathrsfs}

\usepackage{lineno}

\usepackage[top=2.5cm, bottom=2.5cm, left=3cm, right=3cm]{geometry} 

\usepackage{graphics,graphicx,epsf,epstopdf,subfigure}
\graphicspath{{./Figures/}}
\ifpdf
\DeclareGraphicsExtensions{.eps,.pdf,.png,.jpg}
\else
\DeclareGraphicsExtensions{.eps}
\fi

\usepackage[breaklinks,colorlinks,linkcolor=black,citecolor=black,urlcolor=black]{hyperref}

\providecommand{\U}[1]{\protect\rule{.1in}{.1in}}

\newtheorem{theorem}{Theorem}[section]
\newtheorem{corollary}[theorem]{Corollary}
\newtheorem{lemma}[theorem]{Lemma}
\newtheorem{proposition}[theorem]{Proposition}
\newtheorem{definition}[theorem]{Definition}
\newtheorem{remark}{Remark}
\newtheorem{assumption}{Assumption}

\numberwithin{equation}{section}

\newcommand{\dist}{\mathrm{dist}}

\newcommand{\conv}{\mathrm{conv}}
\newcommand{\tr}{\mathrm{tr}}

\newcommand{\proj}{\mathrm{Proj}}

\newcommand{\bR}{\mathbb{R}}
\newcommand{\bN}{\mathbb{N}}

\newcommand{\cC}{\mathcal{C}}

\newcommand{\cM}{\mathcal{M}}
\newcommand{\cN}{\mathcal{N}}
\newcommand{\cO}{\mathcal{O}}
\newcommand{\cR}{\mathcal{R}}
\newcommand{\cT}{\mathcal{T}}

\newcommand{\tG}{\mathtt{G}}
\newcommand{\tV}{\mathtt{V}}
\newcommand{\tE}{\mathtt{E}}

\newcommand{\Rnp}{\mathbb{R}^{n \times p}} 
 
\newcommand{\Rpp}{\mathbb{R}^{p \times p}} 
  
\newcommand{\Rnm}{\mathbb{R}^{n \times m}}

\newcommand{\barX}{\bar{X}}
\newcommand{\hatX}{{\hat{X}}}

\newcommand{\bfone}{\mathbf{1}}
\newcommand{\bfzero}{\mathbf{0}}
\newcommand{\bfW}{\mathbf{W}}
\newcommand{\bfJ}{\mathbf{J}}
\newcommand{\bfD}{\mathbf{D}}

\newcommand{\bfX}{\mathbf{X}}

\newcommand{\bfS}{\mathbf{S}}

\newcommand{\Xik}{X_{i}^{(k)}}
\newcommand{\Xjk}{X_{j}^{(k)}}
\newcommand{\Xikn}{X_{i}^{(k+1)}}

\newcommand{\Dik}{D_{i}^{(k)}}
\newcommand{\Djk}{D_{j}^{(k)}}
\newcommand{\Dikn}{D_{i}^{(k+1)}}

\newcommand{\Sik}{S_{i}^{(k)}}

\newcommand{\Eik}{E_{i}^{(k)}}

\newcommand{\bbXk}{\bar{\mathbf{X}}^{(k)}}
\newcommand{\bbXkn}{\bar{\mathbf{X}}^{(k+1)}}

\newcommand{\avXk}{\hat{\mathbf{X}}^{(k)}}

\newcommand{\avDk}{\hat{\mathbf{D}}^{(k)}}
\newcommand{\avDkn}{\hat{\mathbf{D}}^{(k+1)}}

\newcommand{\avGk}{\hat{\mathbf{G}}^{(k)}}

\newcommand{\bfXk}{\mathbf{X}^{(k)}}
\newcommand{\bfXkn}{\mathbf{X}^{(k+1)}}

\newcommand{\bfDk}{\mathbf{D}^{(k)}}
\newcommand{\bfDkn}{\mathbf{D}^{(k+1)}}

\newcommand{\bfGk}{\mathbf{G}^{(k)}}
\newcommand{\bfGkn}{\mathbf{G}^{(k+1)}}

\newcommand{\bfSk}{\mathbf{S}^{(k)}}

\newcommand{\barXk}{\bar{X}^{(k)}}
\newcommand{\barXkn}{\bar{X}^{(k + 1)}}

\newcommand{\hatXk}{\hat{X}^{(k)}}
\newcommand{\hatXkn}{\hat{X}^{(k + 1)}}

\newcommand{\hatDk}{\hat{D}^{(k)}}
\newcommand{\hatDkn}{\hat{D}^{(k + 1)}}

\newcommand{\hatSk}{\hat{S}^{(k)}}

\newcommand{\hatGk}{\hat{G}^{(k)}}

\newcommand{\zz}{^{\top}}

\newcommand{\ff}{_{\mathrm{F}}}
\newcommand{\fs}{^2_{\mathrm{F}}}

\newcommand{\Snp}{\mathcal{S}_{n,p}}

\newcommand{\dkh}[1]{\left(#1\right)}
\newcommand{\hkh}[1]{\left\{#1\right\}}

\newcommand{\jkh}[1]{\left\langle#1\right\rangle}
\newcommand{\norm}[1]{\left\|#1\right\|}
\newcommand{\abs}[1]{\left\lvert #1\right\rvert}

\newcommand{\prox}{\mathrm{Prox}}

\newcommand{\iid}{i \in [d]}

\newcommand{\sumiid}{\sum_{i=1}^d}

\newcommand{\sumjjd}{\sum\limits_{j=1}^d}

\allowdisplaybreaks
\graphicspath{ {./results/} }
\definecolor{Gray}{rgb}{0.5,0.5,0.5}
\DeclareMathOperator*{\argmin}{arg\,min}

\makeatletter

\newcommand{\Rmnum}[1]{\expandafter\@slowromancap\romannumeral #1@}
\makeatother

\begin{document}


\title{A Decentralized Proximal Gradient Tracking Algorithm  
for Composite Optimization on Riemannian Manifolds}

\author{Lei Wang\thanks{Department of Statistics, Pennsylvania State University, University Park, PA, USA (\href{mailto:lzw5569@psu.edu}{lzw5569@psu.edu}).} 
\and Le Bao\thanks{Department of Statistics, Pennsylvania State University, University Park, PA, USA (\href{mailto:lub14@psu.edu}{lub14@psu.edu}).}
\and Xin Liu\thanks{State Key Laboratory of Scientific and Engineering Computing, Academy of Mathematics and Systems Science, Chinese Academy of Sciences, and University of Chinese Academy of Sciences, Beijing, China (\href{mailto:liuxin@lsec.cc.ac.cn}{liuxin@lsec.cc.ac.cn}).}}

\date{} 

\maketitle

\begin{abstract}

This paper focuses on minimizing a smooth function combined with a nonsmooth regularization term on a compact Riemannian submanifold embedded in the Euclidean space under a decentralized setting.
Typically, there are two types of approaches at present for tackling such composite optimization problems.
The first, subgradient-based approaches, rely on subgradient information of the objective function to update variables, achieving an iteration complexity of $\cO(\epsilon^{-4}\log^2(\epsilon^{-2}))$.
The second, smoothing approaches, involve constructing a smooth approximation of the nonsmooth regularization term, resulting in an iteration complexity of $\cO(\epsilon^{-4})$.
This paper proposes a proximal gradient type algorithm that fully exploits the composite structure. 
The global convergence to a stationary point is established with a significantly improved iteration complexity of $\cO(\epsilon^{-2})$.
To validate the effectiveness and efficiency of our proposed method, we present numerical results in real-world applications, showcasing its superior performance.

\end{abstract}


\section{Introduction}
	
Given a set of $d \in \bN$ agents connected by a communication network, our focus is on the composite optimization problems on the Riemannian manifold of the following form:
\begin{equation} \label{opt:stiefel}
	\min_{X \in \cM} \hspace{2mm} 
	\dfrac{1}{d} \sumiid f_i(X) + r (X),
\end{equation}
where $f_i: \Rnp \to \bR$ is a smooth local function privately owned by agent $i$, $r: \Rnp \to \bR$ is a convex function known to all the agents, and $\cM$ is a compact Riemannian submanifold embedded in $\Rnp$ \cite{Absil2008}.
Employing a regularizer allows us to use prior knowledge about the problem structure explicitly.
Throughout this paper, we make the following assumptions on the objective functions.

\begin{assumption}\label{asp:objective}
	The following statements hold in the problem \eqref{opt:stiefel}.
	
	\begin{enumerate}
		
		
		\item $f_i$ is smooth and its gradient $\nabla f_i$ is Lipschitz continuous over the convex hull of $\cM$, denoted by $\conv(\cM)$, with the corresponding Lipschitz constant $L_f > 0$.
		
		\item $r$ is convex and Lipschitz continuous with the corresponding Lipschitz constant $L_r > 0$. 
		Moreover, the proximal mapping $\prox_{\lambda r} (X)$ of $r$, 
		which is defined by
		\begin{equation*}
			\prox_{\lambda r} (X) := \argmin_{Y \in \Rnp}\; r(Y) + \dfrac{1}{2\lambda} \norm{Y - X}\fs,
		\end{equation*}
		is easy-to-compute for any $\lambda > 0$ and $X \in \Rnp$.
		
	\end{enumerate}
\end{assumption}

Problems of the form \eqref{opt:stiefel} find significant applications across various scientific and engineering fields, 
such as sparse principal component analysis (PCA) \cite{Jolliffe2003,Wang2023communication}, 
subspace learning \cite{Mishra2019riemannian},
matrix completion \cite{Boumal2015},
orthogonal dictionary learning \cite{Zhu2019linearly,Li2021weakly},
deep neural networks with batch normalization \cite{Cho2017riemannian}.
However, under the decentralized setting, designing efficient algorithms for these problems becomes particularly challenging. 
The complexity arises primarily from the combination of nonsmooth nature of objective functions and nonconvexity of manifold constraints.

\subsection{Network Setting}

We consider a scenario in which the agents can only exchange information with their immediate neighbors.
The network $\tG = (\tV, \tE)$ captures the communication links diffusing information among the agents.
Here, $\tV = [d] := \{1, 2, \dotsc, d\}$ is composed of all the agents and $\tE = \{(i, j) \mid i \text{~and~} j \text{~are connected}\}$ represents the set of communication links.
Throughout this paper, we make the following assumptions on the network.

\begin{assumption} \label{asp:network}
	The communication network $\tG = (\tV, \tE)$ is connected.
	Furthermore, there exists a mixing matrix $W = [W(i, j)] \in \bR^{d \times d}$ associated with $\tG$ satisfying the following conditions.
	\begin{enumerate}
		
		\item $W$ is symmetric and nonnegative.
		
		\item $W \mathbf{1}_d = W\zz \mathbf{1}_d = \mathbf{1}_d$.
		
		\item $W(i, j) = 0$ if $i \neq j$ and $(i, j) \notin \tE$.
		
	\end{enumerate}
\end{assumption}

The assumptions about the mixing matrix are standard in the literature \cite{Nedic2018network}, under which $W$ is primitive and doubly stochastic and conforms to the underlying network topology.
Then invoking the Perron-Frobenius Theorem \cite{Pillai2005perron}, we know that the eigenvalues of $W$ lie in $(-1, 1]$ and
\begin{equation} \label{eq:sigma}
	\sigma := \norm{W - \bfone_d \bfone_d\zz / d}_2 < 1.
\end{equation}
The parameter $\sigma$ characterizes the connectivity of the network $\tG$ and plays a prominent part in the analysis of decentralized methods.
In fact, $\sigma = 0$ indicates the full connectivity of $\tG$, while $\sigma$ approaches $1$ as the connectivity of $\tG$ worsens.

Finally, it is noteworthy that the mixing matrix $W$ in Assumption \ref{asp:network}, which always exists, can be efficiently constructed via the exchange of local degree information of $\tG$ on a neighbor-to-neighbor basis.
Interested readers can refer to \cite{Xiao2004,Gharesifard2012distributed,Shi2015,Nedic2018network} for more details.
	
\subsection{Literature Survey}

Recent years have witnessed the repaid development of decentralized optimization in the Euclidean space, marked by the emergence of various algorithms for different types of problems, 
such as decentralized gradient descent algorithms \cite{Nedic2009,Yuan2016,Zeng2018},
gradient tracking methods \cite{Xu2015,Qu2017,Nedic2017,Sun2022distributed,Song2023optimal}, 
primal-dual frameworks \cite{Shi2015,Ling2015,Chang2015multi,Hajinezhad2019}, 
proximal gradient tracking approaches \cite{Di2016,Scutari2019,Xin2021stochastic,Yan2023compressed}, 
decentralized Newton methods \cite{Bajovic2017newton,Zhang2021newton,Daneshmand2021newton}, 
and so on.
Interested readers can refer to some recent surveys \cite{Nedic2018network,Xin2020general,Chang2020distributed} and references therein for a complete review of the decentralized algorithms in the Euclidean space.

The investigation of decentralized algorithms for optimization problems on Riemannian manifolds is still in its infancy. 
Existing studies, predominantly centered on the particular case of Stiefel manifolds, can be categorized into two main groups, which will be briefly introduced below.

The first group of algorithms extends classical decentralized methods in the Euclidean space by leveraging the geometric tools derived from Riemannian optimization \cite{Absil2008}.
For instance, Chen et al. \cite{Chen2021decentralized} propose a decentralized Riemannian stochastic gradient descent (DRSGD) method along with its gradient-tracking variant DRGTA.
Subsequently, these two methods are generalized to compact Riemannian submanifolds in \cite{Deng2023decentralized}, resulting in algorithms named DPRGD and DPRGT, respectively.
Moreover, the algorithm DRCGD, developed in \cite{Chen2023decentralized}, incorporates the Riemannian conjugate gradient method into the framework of \cite{Deng2023decentralized} to enhance convergence rates.
Additionally, Deng et al. \cite{Deng2023douglas} propose a decentralized algorithm DDRS method together with its inexact version iDDRS based on the Douglas--Rachford splitting method.
The aforementioned approaches are tailored for smooth objective functions and, therefore, may not be directly applicable to the problem \eqref{opt:stiefel}.
Recently, Wang et al. \cite{Wang2023decentralized} have introduced a decentralized Riemannian subgradient method (DRSM) aimed at tackling nonsmooth optimization problems on Riemannian manifolds. This method operates under the assumption that each local function is weakly convex.
However, it is worth noting that DRSM suffers from a slow convergence rate, attributed to its reliance solely on subgradient information.

The second group of algorithms utilizes penalty functions to handle nonconvex manifold constraints, including DESTINY \cite{Wang2022decentralized}, VRSGT \cite{Wang2023variance}, and THANOS \cite{Wang2023smoothing}.
Specifically, DESTINY incorporates a gradient tracking scheme into the minimization of an approximate augment Lagrangian function.
VRSGT goes a step further by integrating the variance reduction technique to simultaneously reduce the communication and sampling complexities.
To deal with nonsmooth regularizers, THANOS constructs an approximation of the objective function based on the Moreau envelope.
It is crucial to emphasize that these algorithms heavily rely on the specific structure of the Stiefel manifold, posing challenges for their extension to more general Riemannian manifolds.

\subsection{Our Contributions}

In response to the growing demand for processing large-scale datasets in practical applications, we propose a novel decentralized algorithm called DR-ProxGT for solving the problem \eqref{opt:stiefel}.
Our approach involves solving a proximal gradient subproblem on the tangent space, where we leverage the gradient tracking technique to estimate the Euclidean gradient across the whole network.
The convergence analysis reveals that DR-ProxGT converges globally to a stationary point of \eqref{opt:stiefel} and exhibits an iteration complexity of $\cO(\epsilon^{-2})$.
To the best of our knowledge, this achieves the best result in the literature.
Please refer to Table~\ref{tb:complexity} for a comparative overview of existing complexity results.
In particular, we establish a novel descent inequality for the nonconvex composite objective function and derive a refined uniform bound of gradient trackers by meticulously handling the proximal mapping. 
These technical results are of independent interest and have potential applications beyond the scope of this specific problem.
Finally, we conduct comprehensive numerical experiments to compare DR-ProxGT with two competing algorithms. 
The test results are strongly in favor of our algorithm in practical applications.

\begin{table}[ht]
\caption{A summary of the iteration complexity for existing algorithms to find an $\epsilon$-stationary point of the problem \eqref{opt:stiefel}.}
\label{tb:complexity}
\begin{tabular*}{\textwidth}{@{\extracolsep{\fill}}cccc@{\extracolsep{\fill}}}
	\toprule%
	Algorithm & Manifold & Stepsize & Iteration Complexity~~ \\
	\midrule
	DRSM \cite{Wang2023decentralized} 
	& Stiefel manifold
	& diminishing stepsize
	& $\cO(\epsilon^{-4}\log^2(\epsilon^{-2}))$ \\
	THANOS \cite{Wang2023smoothing} 
	& Stiefel manifold
	& fixed stepsize
	& $\cO(\epsilon^{-4})$ \\
	~~DR-ProxGT (this work) 
	& compact manifold
	& fixed stepsize
	& $\cO(\epsilon^{-2})$ \\
	\bottomrule
\end{tabular*}
\end{table}

\subsection{Notations}

The following notations are adopted throughout this paper.
The Euclidean inner product of two matrices \(Y_1, Y_2\) with the same size is defined as \(\jkh{Y_1, Y_2}=\tr(Y_1\zz Y_2)\), where $\tr (B)$ stands for the trace of a square matrix $B$.
And the notation $I_p \in \Rpp$ represents the $p \times p$ identity matrix.
The Frobenius norm and 2-norm of a given matrix \(C\) are denoted by \(\norm{C}\ff\) and \(\norm{C}_2\), respectively. 
The $(i, j)$-th entry of a matrix $C$ is represented by $C (i, j)$.
The notation $\bfone_d \in \bR^d$ and $\bfzero_d \in \bR^d$ stand for the $d$-dimensional vector of all ones and all zeros, respectively.
We define the distance and the projection of a point $X \in \Rnp$ onto a set $\cC \subset \Rnp$ by $\dist(X, \cC) := \inf\{\norm{Y - X}\ff \mid Y \in \cC\}$ and $\proj_{\cC}(X) := \argmin_{Y \in \cC}\norm{Y - X}\ff$, respectively.
The Kronecker product is denoted by $\otimes$.
Given a differentiable function \(g(X) : \Rnp \to \bR\), the Euclidean gradient of \(g\) with respect to \(X\) is represented by \(\nabla g(X)\).
Further notations will be introduced wherever they occur.

\subsection{Outline}

The remainder of this paper is organized as follows.
Section \ref{sec:preliminaries} draws into some preliminaries of Riemannian optimization.
In Section \ref{sec:algorithm}, we devise a proximal gradient type algorithm for solving the problem \eqref{opt:stiefel}.  
The convergence properties of the proposed algorithm are investigated in Section \ref{sec:convergence}.
Numerical results are presented in Section \ref{sec:experiments} to evaluate the performance of our algorithm.
Finally, this paper concludes with concluding remarks and key insights in Section \ref{sec:conclusions}.

\section{Preliminaries} \label{sec:preliminaries}

This section introduces and reviews some basic notions and concepts regarding Riemannian manifolds that are closely related to the present work in this paper.

\subsection{Stationarity Condition}

In this subsection, we delve into the stationarity condition of the problem \eqref{opt:stiefel}.
Towards this end, we first introduce the definition of Clarke subgradient \cite{Clarke1990} for nonsmooth functions.

\begin{definition}
	Suppose $f: \Rnp \to \bR$ is a Lipschitz continuous function.
	The generalized directional derivative of $f$ at the point $X \in \Rnp$
	along the direction $H \in  \Rnp$ is defined by:
	\begin{equation*}
		f^{\circ} (X; H) := \limsup\limits_{Y \to X,\, t \to 0^+} \dfrac{f (Y + t H) - f(Y)}{t}.
	\end{equation*}
	Based on generalized directional derivative of $f$,
	the (Clarke) subgradient of $f$ is defined by:
	\begin{equation*}
		\partial f(X) := \{G \in \Rnp \mid \jkh{G, H} \leq f^{\circ} (X; H) \}.
	\end{equation*}
\end{definition}

Additionally, we employ geometric concepts of Riemannian manifolds to present the stationarity condition.
For each point $X \in \cM$, the tangent space to $\cM$ at $X$ is referred to as $\cT_{X}$, and the notation $\cN_{X}$ represents the normal space at $X$. 
In this paper, we consider the Riemannian metric $\jkh{\cdot, \cdot}_X$ on $\cT_{X}$ that is induced from the Euclidean inner product $\jkh{\cdot, \cdot}$, i.e., $\jkh{V_1, V_2}_X = \jkh{V_1, V_2} = \tr(V_1\zz V_2)$.
Roughly speaking, the tangent space $\cT_{X}$ intuitively contains the possible directions in which one can tangentially pass through $X$, while the normal space $\cN_{X}$ is the orthogonal complement of $\cT_{X}$ in $\Rnp$.
For precise statements on these notions, interested readers can refer to the monograph \cite{Absil2008}.

For convenience, we denote $f(X) := \sumiid f_i(X) / d$ and $h(X) := f(X) + r(X)$.
Now we are prepared to state the stationarity condition of the problem \eqref{opt:stiefel}, which has been thoroughly discussed in \cite{Yang2014,Chen2020}.

\begin{definition}
	A point $X \in \cM$ is called a stationary point of the problem \eqref{opt:stiefel} if it satisfies the following stationarity condition,
	\begin{equation*}
		0 \in \proj_{\cT_{X}}\dkh{\partial h(X)} = \proj_{\cT_{X}}\dkh{\nabla f(X) + \partial r(X)},
	\end{equation*}
	or equivalently, there exists an element $R(X) \in \partial r(X)$ such that $\nabla f(X) + R(X) \in \cN_{X}$.
\end{definition}

\subsection{Proximal Smoothness}

Our theoretical analysis heavily relies on the concept of proximal smoothness, which tides us over the obstacle incurred by the nonconvexity of manifolds.
Following \cite{Clarke1995proximal}, we say a closed set $\cC$ is $\delta$-proximally smooth for a constant $\delta > 0$ if the projection $\proj_{\cC}(X)$ is a singleton whenever $\dist(X, \cC) < \delta$.
It should be noted that the operation operator $\proj_{\cC}(X)$ of a $\delta$-proximally smooth set $\cC$ is Lipschitz continuous as long as $X$ is not far away from $\cC$.
Specifically, for any $\gamma \in (0, \delta)$, the following relationship holds whenever $\dist(X, \cC) < \gamma$ and $\dist(Y, \cC) < \gamma$,
\begin{equation*}
	\norm{\proj_{\cC}(X) - \proj_{\cC}(Y)}\ff \leq \dfrac{\delta}{\delta - \gamma} \norm{X - Y}\ff.
\end{equation*}

As is shown in \cite{Clarke1995proximal,Davis2020stochastic,Balashov2021gradient}, any compact $C^2$-manifold $\cM$ embedded in the Euclidean space is proximally smooth.
For instance, the Stiefel manifold is $1$-proximally smooth and the Grassmann manifold is $1 / \sqrt{2}$-proximally smooth.
For ease of notation, we assume that the manifold $\cM$ is $2 \gamma$-proximally smooth for a constant $\gamma > 0$.
Then for any $X, Y \in \cR(\gamma) := \{X \in \Rnp \mid \dist(X, \cM) < \gamma\}$, we have
\begin{equation} \label{eq:proj-Lip}
	\norm{\proj_{\cM}(X) - \proj_{\cM}(Y)}\ff \leq 2 \norm{X - Y}\ff.
\end{equation}

Below is another crucial inequality regarding the projection operator $\proj_{\cM}$ indicating that $X +  \proj_{\cT_{X}}(V)$ is a second-order approximation of $\proj_{\cM}(X + V)$ for any $X \in \cM$ and $V \in \Rnp$.
Specifically, it holds that
\begin{equation} \label{eq:proj-Lip-s}
	\norm{\proj_{\cM}(X + V) - X -  \proj_{\cT_{X}}(V)}\ff \leq M_{pj} \norm{V}\fs,
\end{equation}
where $M_{pj} > 0$ is a constant.
We refer interested readers to \cite{Deng2023decentralized} for the construction of \eqref{eq:proj-Lip-s}.

\subsection{Consensus on Riemannian Manifolds}

In decentralized networks, where only local communications are permitted, each agent $i$ has to maintain its local copy $X_i \in \cM$ of the common variable $X$ in the problem \eqref{opt:stiefel}. 
Let $\hatX$ be the Euclidean average of $\{X_i\}_{i = 1}^d$ defined by
\begin{equation*}
	\hatX := \dfrac{1}{d} \sumiid X_i 
	= \argmin_{Y \in \Rnp} \sumiid \norm{Y - X_i}\fs.
\end{equation*}
To access the consensus error, the quantity $\sumiid \|X_i - \hatX\|\fs$ is typically used in the convergence analysis of Euclidean decentralized algorithms.

It should be noted that the average $\hatX$ may not necessarily lie in the manifold $\cM$ even if $X_i \in \cM$ for $\iid$, owing to the nonconvexity of $\cM$. 
This observation motivates the definition of induced arithmetic mean introduced in \cite{Sarlette2009consensus} as follows,
\begin{equation*}
	\barX \in \argmin_{Y \in \cM} \sumiid \norm{Y - X_i}\fs
	= \proj_{\cM}(\hatX).
\end{equation*}
According to the stationarity condition of the above problem, we have $\barX - \hatX \in \cN_{\barX}$.
In the development of our algorithm, we will guarantee that $\hatX \in \cR(\gamma)$.
The proximal smoothness of $\cM$ then results in that $\proj_{\cM}(\hatX)$ is a singleton.

Moreover, as is discussed in \cite{Deng2023decentralized}, the distance between $\hatX$ and $\barX$ can be controlled by the consensus error measured by $\barX$ under a certain condition.
Specifically, for any $\{X_i \in \cM\}_{i = 1}^d$ satisfying $\max_{\iid}\|X_i - \barX\|\ff \leq \gamma$, there exists a constant $M_{av} > 0$ such that
\begin{equation} \label{eq:dist-aver}
	\norm{\hatX - \barX}\ff \leq \dfrac{M_{av}}{d} \sumiid \norm{X_i - \barX}\fs,
\end{equation}
which will be used in the subsequent analysis.

\section{Decentralized Riemannian Proximal Gradient Tracking} \label{sec:algorithm}

In this section, we devise an efficient decentralized algorithm to solve the problem \eqref{opt:stiefel} by exploiting the composite structure.
Each iteration of our algorithm attempts to tackle a convex subproblem coupled with a linear constraint.
The objection function of this subproblem is constructed by linearizing the function $f$ around the current iterate.
Simultaneously, the linear constraint in question captures the structure of the tangent space.
A key feature of our algorithm is the integration of the gradient tracking technique, which is employed to estimate $\nabla f$ across the entire network. 
With the help of gradient tracking, our algorithm can effectively achieve exact consensus.

We introduce two auxiliary local variables, $D_i \in \Rnp$ and $S_i \in \Rnp$, for agent $i$ in our algorithm.
Specifically, $D_i$ is designed to track the global gradient $\nabla f$ through the exchange of local gradient information, while $S_i$ aims at estimating the search direction on the tangent space based on $D_i$.

Hereafter, we use the notations $\Xik$, $\Dik$, and $\Sik$ to represent the $k$-th iterate of $X_i$, $D_i$, and $S_i$, respectively.
The key steps of our algorithm from the perspective of each agent are outlined below.

\paragraph{Step 1: $S$-update.}
Given the composite structure, a natural approach to tackle the problem \eqref{opt:stiefel} involves evaluating the following proximal gradient step constrained to the tangent space, as proposed by \cite{Chen2020},
\begin{equation*}
	\min_{S \in \cT_{X}} \hspace{2mm} \jkh{\nabla f (X), S} + \dfrac{1}{2 \tau} \norm{S}\fs + r (X + S), 
\end{equation*}
where $\tau > 0$ is a constant.
The intuition behind this method is to seek a descent direction on the tangent space by replacing the smooth term $f$ with its first-order approximation around the current estimate.
Under the decentralized setting, however, the global gradient $\nabla f$ is not available.
To overcome this challenge, we introduce the auxiliary variable $D_i \in \Rnp$ at agent $i$ to estimate $\nabla f$ across the whole network.
Specifically, at iteration $k$, each agent $i$ aims to solve the following subproblem to obtain the search direction,
\begin{equation} \label{eq:update-S}
	\Sik := \argmin_{S_i \in \cT_{\Xik}} \hspace{2mm}
	\jkh{\Dik, S_i} + \dfrac{1}{2 \tau} \norm{S_i}\fs + r (\Xik + S_i),
\end{equation}
where $\tau > 0$ remains a constant.
The subproblem in \eqref{eq:update-S} involves minimizing a strongly convex function subject to linear constraints.
In the specific scenario where $\cM = \Snp := \{X \in \Rnp \mid X\zz X = I_p\}$, Chen et al. \cite{Chen2020} and Xiao et al. \cite{Xiao2021penalty} have utilized the semi-smooth Newton method and the Uzawa method, respectively, to address a subproblem with a comparable structure. 
Notably, these methodologies can be seamlessly extended to encompass the case of compact manifolds embedded in the Euclidean space.

\paragraph{Step 2: $X$-update.}
Once the search direction is determined, our algorithm executes a local update along the direction of $\Sik$ for agent $i$ at iteration $k$, combining the consensus step with the projection onto $\cM$,
\begin{equation} \label{eq:update-X}
	\Xikn := \proj_{\cM} \dkh{\sumjjd W^t(i, j) \Xjk + \eta \Sik},
\end{equation}
where $\eta > 0$ is a stepsize, and $t \geq 1$ is an integer.
To guarantee the convergence, we incorporate multiple consensus steps to update local variables. 
This requirement is a common practice for achieving consensus on Riemannian manifolds \cite{Chen2021decentralized,Wang2023decentralized,Deng2023decentralized,Chen2023decentralized,Deng2023douglas}.

\paragraph{Step 3: $D$-update.}
Finally, leveraging the updated information, each agent $i$ computes the new estimate of $\nabla f$ by resorting to the gradient tracking technique \cite{Zhu2010discrete,Daneshmand2020},
\begin{equation} \label{eq:update-D}
	\Dikn := \sumjjd W^t(i, j) \Djk + \nabla f_i (\Xikn) - \nabla f_i (\Xik).
\end{equation}
Diverging from existing approaches \cite{Chen2021decentralized, Deng2023decentralized}, our algorithm directly tracks the Euclidean gradient instead of the Riemannian gradient \cite{Absil2008}.
This strategy serves to alleviate the computational burden associated with projecting the Euclidean gradient onto the tangent space.

The whole procedure is summarized in Algorithm \ref{alg:DR-ProxGT},
named {\it Decentralized Riemannian proximal gradient tracking}
and abbreviated to DR-ProxGT.
Our approach merges the Riemannian proximal gradient method with the gradient tracking technique, resulting in a substantial enhancement in iteration complexity. The details will be elaborated upon in the subsequent section.


\begin{algorithm}[ht!]
	\caption{Decentralized Riemannian proximal gradient tracking (DR-ProxGT).} 
	\label{alg:DR-ProxGT}
	
	\KwIn{$X_{\mathrm{initial}} \in \cM$, $t \geq 1$, $\tau > 0$, and $\eta > 0$.}
	
	Set $k := 0$.
	
	\For{$\iid$}{
	
		Initialize $\Xik := X_{\mathrm{initial}}$
		and $\Dik := \nabla f_i (\Xik)$.
	
	}

	\While{``not converged''}
	{
	
		\For{$\iid$}{
			
			Compute $\Sik$ by \eqref{eq:update-S}.
			
			Update $\Xikn$ by \eqref{eq:update-X}.
			
			Update $\Dikn$ by \eqref{eq:update-D}.
		
		}
		
		Set $k := k + 1$.
		
	}
	
	\KwOut{$\{\Xik\}_{i = 1}^d$.}
	
\end{algorithm}

\section{Convergence Analysis} \label{sec:convergence}

The global convergence of our proposed Algorithm \ref{alg:DR-ProxGT} is rigorously established under mild conditions in this section.
To facilitate the narrative, we define the following notations.
\begin{itemize}
	
	\item $J = \bfone_d \bfone_d\zz / d$,
	$\bfJ = J \otimes I_n$,
	$\bfW^t = W^t \otimes I_n$.
	
	\item $\hatXk = \dfrac{1}{d} \sum\limits_{i = 1}^d \Xik$, 
	$\barXk \in \proj_{\cM}(\hatXk)$,
	$\hatDk = \dfrac{1}{d} \sum\limits_{i = 1}^d \Dik$,
	$\hatGk = \dfrac{1}{d} \sum\limits_{i = 1}^d \nabla f_i(\Xik)$.
	
	\item $\bfXk = [(X_1^{(k)})\zz, \dotsc, (X_d^{(k)})\zz]\zz$,
	$\avXk = (\bfone_d \otimes I_n) \hatXk = \bfJ \bfXk$,
	$\bbXk = (\bfone_d \otimes I_n) \barXk$.
	
	\item $\bfDk = [(D_1^{(k)})\zz, \dotsc, (D_d^{(k)})]\zz$,
	$\avDk = (\bfone_d \otimes I_n) \hatDk = \bfJ \bfDk$.
	
	\item $\bfGk = [(\nabla f_1(X_1^{(k)}))\zz, \dotsc, (\nabla f_d(X_d^{(k)}))]\zz$,
	$\avGk = (\bfone_d \otimes I_n) \hatGk = \bfJ \bfGk$.
	
	\item $\bfSk = [(S_1^{(k)})\zz, \dotsc, (S_d^{(k)})]\zz$,
	$\hatSk = \dfrac{1}{d} \sum\limits_{i = 1}^d \Sik$.
	
\end{itemize}

\subsection{Boundedness of Iterates}

The purpose of this subsection is to show the boundedness of the iterate sequence $\{(\bfXk, \bfDk, \bfSk)\}$ generated by Algorithm \ref{alg:DR-ProxGT}.
It is obvious that the sequence $\{\bfXk\}$ is bounded since $\cM$ is a compact manifold.
Next, we prove that the sequence $\{\bfDk\}$ is bounded.

\begin{lemma} \label{le:bound-D}
	Suppose that Assumption \ref{asp:objective} and Assumption \ref{asp:network} hold.
	Let $\{(\bfXk, \bfDk, \bfSk)\}$ be the iterate sequence generated by Algorithm \ref{alg:DR-ProxGT} with $t > \lceil \log_{\sigma}(1/2) \rceil$.
	Then for any $k \in \bN$, it holds that 
	\begin{equation*}
		\norm{\bfDk}\ff \leq 6 \sqrt{d} M_g,
	\end{equation*}
	where $M_g := \sup \{\norm{\nabla f_i (X)}\ff \mid X \in \cM, \iid\} > 0$ is a constant.
\end{lemma}

\begin{remark}
	Since $\nabla f_i$ is Lipschitz continuous and $\cM$ is compact, the constant $M_g$ is well-defined.
\end{remark}

\begin{proof}
	According to the triangular inequality, we have
	\begin{equation*}
		\norm{\bfGkn - \bfGk}\ff 
		\leq \norm{\bfGkn}\ff + \norm{\bfGk}\ff 
		\leq 2 \sqrt{d} M_g,
	\end{equation*}
	where the last inequality follows from the definition of $M_g$.
	And straightforward calculations gives rise to the following relationship,
	\begin{equation*}
	\begin{aligned}
		\norm{\bfDkn - \avDkn}\ff
		 = {} & \norm{ \dkh{\bfW^t - \bfJ} \dkh{\bfDk - \avDk} 
		 + \dkh{I_{dn} - \bfJ} \dkh{\bfGkn - \bfGk} }\ff \\
		 \leq {} & \norm{\dkh{\bfW^t - \bfJ} \dkh{\bfDk - \avDk}}\ff
		 + \norm{\bfGkn - \bfGk}\ff \\
		 \leq {} & \sigma^t \norm{\bfDk - \avDk}\ff + 2 \sqrt{d} M_g.
	\end{aligned}
	\end{equation*}
	Then, by mathematical induction, we can obtain that
	\begin{equation*}
		\norm{\bfDkn - \avDkn}\ff 
		\leq (\sigma^t)^{k + 1} \norm{\mathbf{D}^{(0)} - \hat{\mathbf{D}}^{(0)}}\ff 
		+ \dfrac{2 \sqrt{d} M_g (1 - (\sigma^t)^{k + 1})}{1 - \sigma^t}.
	\end{equation*}
	which together with $\mathbf{D}^{(0)} = \mathbf{G}^{(0)}$ and $\hat{\mathbf{D}}^{(0)} = \hat{\mathbf{G}}^{(0)}$ yields that
	\begin{equation*}
		\norm{\bfDkn - \avDkn}\ff 
		\leq (\sigma^t)^{k + 1} \norm{\mathbf{G}^{(0)} - \hat{\mathbf{G}}^{(0)}}\ff 
		+ \dfrac{2 \sqrt{d} M_g (1 - (\sigma^t)^{k + 1})}{1 - \sigma^t}
		\leq \dfrac{2 \sqrt{d} M_g (1 - (\sigma^t)^{k + 2})}{1 - \sigma^t}.
	\end{equation*}
	In addition, it can be straightforwardly verified that
	\begin{equation} \label{eq:bound-hatD}
		\norm{\hatDk}\ff 
		= \norm{\dfrac{1}{d} \sumiid \nabla f_i(\Xik)}\ff
		\leq \dfrac{1}{d} \sumiid \norm{\nabla f_i(\Xik)}\ff
		\leq M_g.
	\end{equation}
	Therefore, we can obtain that
	\begin{equation*}
	\begin{aligned}
		\norm{\bfDkn}\ff 
		& \leq \norm{\bfDkn - \avDkn}\ff + \norm{\avDkn}\ff
		= \norm{\bfDkn - \avDkn}\ff + \sqrt{d} \norm{\hatDkn}\ff \\
		& \leq \dfrac{2 \sqrt{d} M_g (1 - (\sigma^t)^{k + 2})}{1 - \sigma^t}
		+ \sqrt{d}M_g
		\leq 6 \sqrt{d} M_g,
	\end{aligned}
	\end{equation*}
	where the last inequality follows from the fact that $\sigma^t \in (0, 1/2)$.
	The proof is completed.
\end{proof}

Then we prove that the norm of $\Sik$ is controlled by that of $\Dik$ for any $\iid$ and $k \in \bN$ in the following lemma.

\begin{lemma} \label{le:bound-Si}
	Suppose that Assumption \ref{asp:objective} and Assumption \ref{asp:network} hold.
	Let $\{(\bfXk, \bfDk, \bfSk)\}$ be the iterate sequence generated by Algorithm \ref{alg:DR-ProxGT}.
	Then, for any $\iid$ and $k \in \bN$, it holds that
	\begin{equation*}
		\norm{\Sik}\ff 
		\leq \tau \norm{\Dik}\ff 
		+ \tau L_r.
	\end{equation*}
\end{lemma}

\begin{proof}
	To begin with, the assertion of this lemma is obvious if $\Sik = 0$.
	Next, we investigate the case that $\Sik \neq 0$.
	For convenience, we denote the objective function in \eqref{eq:update-S} by
	\begin{equation*}
		g_i^{(k)} (S) := \jkh{\Dik, S} + \dfrac{1}{2 \tau} \norm{S}\fs + r (\Xik + S).
	\end{equation*}
	Since $g_i^{(k)}$ is strongly convex with modulus $1 / \tau$, we have
	\begin{equation} \label{eq:gik}
		g_i^{(k)} (S^{\prime}) 
		\geq g_i^{(k)} (S) 
		+ \jkh{\partial g_i^{(k)} (S), S^{\prime} - S} 
		+ \dfrac{1}{2 \tau} \norm{S^{\prime} - S}\fs,
	\end{equation}
	for any $S, S^{\prime} \in \Rnp$.
	In particular, if $S, S^{\prime} \in \cT_{\Xik}$, it holds that
	\begin{equation*}
		\jkh{\partial g_i^{(k)} (S), S^{\prime} - S}
		= \jkh{\proj_{\cT_{\Xik}} \dkh{\partial g_i^{(k)} (S)}, S^{\prime} - S}.
	\end{equation*}
	Then it follows from the first-order optimality condition of \eqref{eq:update-S} that
	\begin{equation*}
		0 \in \proj_{\cT_{\Xik}} \dkh{\partial g_i^{(k)} (\Sik)}.
	\end{equation*}
	Upon taking $S = \Sik$ and $S^{\prime} = 0$ in \eqref{eq:gik}, we can obtain that
	\begin{equation*}
		g_i^{(k)} (0) - g_i^{(k)} (\Sik)
		\geq \dfrac{1}{2 \tau}\norm{\Sik}\fs,
	\end{equation*}
	which together with the Lipschitz continuity of $r$ infers that
	\begin{equation*}
		\begin{aligned}
			\dfrac{1}{\tau}\norm{\Sik}\fs
			\leq r (\Xik)  - r (\Xik + \Sik)
			- \jkh{\Dik, \Sik}
			\leq L_r \norm{\Sik}\ff
			+ \norm{\Dik}\ff \norm{\Sik}\ff.
		\end{aligned}
	\end{equation*}
	Hence, we can arrive at the conclusion that
	\begin{equation*}
		\norm{\Sik}\ff 
		\leq \tau \norm{\Dik}\ff 
		+ \tau L_r,
	\end{equation*}
	as desired.
\end{proof}

Now the boundedness of the sequence $\{\bfSk\}$ can be straightforwardly obtained by combining the above two lemmas.

\begin{corollary} \label{cor:bound-S}
	Let $\{(\bfXk, \bfDk, \bfSk)\}$ be the iterate sequence generated by Algorithm \ref{alg:DR-ProxGT} with $t > \lceil \log_{\sigma}(1/2) \rceil$.
	Suppose that Assumption \ref{asp:objective} and Assumption \ref{asp:network} hold.
	Then for any $k \in \bN$, we have
	\begin{equation*}
		\norm{\bfSk}\fs 
		\leq 2 d \tau^2 \dkh{36 M_g^2 + L_r^2}.
	\end{equation*}
\end{corollary}

\begin{proof}
	This is a direct consequence of Lemma \ref{le:bound-D} and Lemma \ref{le:bound-Si},
	and the proof is omitted here.
\end{proof}

\subsection{Consensus and Tracking Errors}

This subsection is devoted to building the upper bound of consensus error $\|\bfXk - \bbXk\|\fs$ and tracking error $\|\bfDk - \avDk\|\fs$.
Towards this end, we first show that, based on the following lemma, $\proj_{\cM}(\hatXk)$ is a singleton and $\barXk = \proj_{\cM}(\hatXk)$ for any $k \in \cN$.

\begin{lemma} \label{le:cN}
	Suppose that Assumption \ref{asp:objective} and Assumption \ref{asp:network} hold.
	Let $\{(\bfXk, \bfDk, \bfSk)\}$ be the iterate sequence generated by Algorithm \ref{alg:DR-ProxGT} with
	\begin{equation} \label{eq:cond-N}
		0 < \eta \leq 1,
		0 < \tau \leq \dfrac{\gamma}{24 (6\sqrt{d}M_g + L_r)},
		\mbox{~and~}
		t >  \max\hkh{\left\lceil \log_{\sigma} \dkh{\dfrac{1}{2}} \right\rceil, \left\lceil \log_{\sigma} \dkh{\dfrac{\gamma}{24 \sqrt{d} M_l}} \right\rceil},
	\end{equation}
	where $M_l := \sup\{\norm{X - Y}\ff \mid X, Y \in \cM\} > 0$ is a constant.
	Then the relationship $\max_{\iid}\|\Xik - \barXk\|\ff \leq \gamma / 2$ holds for any $k \in \bN$.
\end{lemma}

\begin{proof}
	We will use mathematical induction to prove this lemma.
	The argument $\max_{\iid}\|X_i^{(0)} - \barX^{(0)}\|\ff \leq \gamma / 2$ directly holds resulting from the initialization. 
	Now, we assume that $\max_{\iid}\|\Xik - \barXk\|\ff \leq \gamma / 2$, and investigate the situation of $\max_{\iid}\|\Xikn - \barXkn\|\ff$.
	
	Our first purpose is to show that
	\begin{equation} \label{eq:cR}
		\sumjjd W^t(i, j) \Xjk + \eta \Sik \in \cR(\gamma), 
	\end{equation}
	for any $\iid$.
	In fact, it can be readily verified that
	\begin{equation*}
	\begin{aligned}
		\norm{\bfXk - \avXk}\fs = \sumiid \norm{\Xik - \hatXk}\fs
		\leq \dfrac{1}{d} \sumiid \sumjjd \norm{\Xik - \Xjk}\fs
		\leq d M_l^2,
	\end{aligned}
	\end{equation*}
	which is followed by
	\begin{equation*}
	\begin{aligned}
		\norm{\sumjjd W^t(i, j) \Xjk - \hatXk}\fs
		\leq {} & \sumiid \norm{\sumjjd W^t(i, j) \Xjk - \hatXk}\fs \\
		= {} & \norm{\bfW^t \bfXk - \avXk}\fs \\
		= {} & \norm{\dkh{\bfW^t - \bfJ}\dkh{\bfXk - \avXk}}\fs \\
		\leq {} & \sigma^{2t} \norm{\bfXk - \avXk}\fs
		\leq d \sigma^{2t} M_l^2.
	\end{aligned}
	\end{equation*}
	Thus, we can obtain that
	\begin{equation*}
	\begin{aligned}
		\norm{\sumjjd W^t(i, j) \Xjk + \eta \Sik - \hatXk}\ff
		\leq {} & \norm{\sumjjd W^t(i, j) \Xjk - \hatXk}\ff
		+ \eta \norm{\Sik}\ff \\
		\leq {} & \sqrt{d} \sigma^t M_l
		+ \eta \tau (6 \sqrt{d} M_g + L_r)
		\leq \dfrac{\gamma}{12}.
	\end{aligned}
	\end{equation*}
	It follows from $\max_{\iid}\|\Xik - \barXk\|\ff \leq \gamma / 2$ that $\|\hatXk - \barXk\|\ff \leq \gamma / 2$, which further yields that
	\begin{equation*}
		\norm{\sumjjd W^t(i, j) \Xjk + \eta \Sik - \barXk}\ff
		\leq \norm{\sumjjd W^t(i, j) \Xjk + \eta \Sik - \hatXk}\ff
		+ \norm{\hatXk - \barXk}\ff
		< \gamma.
	\end{equation*}
	Since $\barXk \in \cM$, we know that the relationship \eqref{eq:cR} holds.
	
	Next, we proceed to prove that $\max_{\iid}\|\Xikn - \barXkn\|\ff \leq \gamma / 2$.
	For any $\iid$, we have
	\begin{equation*}
	\begin{aligned}
		\norm{\Xikn - \barXkn}\ff
		\leq {} & \norm{\Xikn - \barXk}\ff + \norm{\barXk - \barXkn}\ff \\
		\leq {} & \norm{\Xikn - \barXk}\ff + \norm{\barXkn - \hatXkn}\ff + \norm{\hatXkn - \barXk}\ff \\
		\leq {} &  \norm{\Xikn - \barXk}\ff + 2 \norm{\hatXkn - \barXk}\ff \\
		\leq {} & 3 \max_{\iid} \norm{\Xikn - \barXk}\ff,
	\end{aligned}
	\end{equation*}
	where the third inequality results from the definition of $\barXkn$.
	Moreover, according to the relationship \eqref{eq:proj-Lip-s}, it follows that
	\begin{equation*}
		\begin{aligned}
			\norm{\Xikn - \barXk}\ff
			= {} & \norm{\proj_{\cM} \dkh{\sumjjd W^t(i, j) \Xjk + \eta \Sik} - \proj_{\cM} \dkh{\hatXk}}\ff \\
			\leq {} & 2 \norm{\sumjjd W^t(i, j) \Xjk + \eta \Sik - \hatXk}\ff 
			\leq \dfrac{\gamma}{6}.
		\end{aligned}
	\end{equation*}
	Hence, we can conclude that $\max_{\iid}\|\Xikn - \barXkn\|\ff \leq \gamma / 2$.
	The proof is completed.
\end{proof}

According to the proof of Lemma \ref{le:cN}, it follows that the relationship \eqref{eq:cR} holds for any $k \in \bN$ in Algorithm \ref{alg:DR-ProxGT} under the condition \eqref{eq:cond-N}.
Moreover, the relationship $\max_{\iid}\|\Xik - \barXk\|\ff \leq \gamma / 2$ directly implies that the inclusion $\hatXk \in \cR(\gamma)$ is valid.
As a consequence, $\proj_{\cM}(\hatXk)$ is a singleton and $\barXk = \proj_{\cM}(\hatXk)$.
Next, we establish the upper bound of consensus errors.

\begin{lemma} \label{le:error-cons}
	Suppose that Assumption \ref{asp:objective} and Assumption \ref{asp:network} hold.
	Let $\{(\bfXk, \bfDk, \bfSk)\}$ be the iterate sequence generated by Algorithm \ref{alg:DR-ProxGT} with the algorithmic parameters satisfying the condition \eqref{eq:cond-N}.
	Then for any $k \in \bN$, it holds that
	\begin{equation*}
		\norm{\bfXkn - \bbXkn}\fs
		\leq 8 \sigma^{2t} \norm{\bfXk - \bbXk}\fs
		+ 8 \eta^2 \norm{\bfSk}\fs.
	\end{equation*}
\end{lemma}

\begin{proof}
	To begin with, we have
	\begin{equation*}
		\norm{\bfW^t \bfXk - \avXk}\fs
		= \norm{\dkh{\bfW^t - \bfJ} \dkh{\bfXk - \bbXk}}\fs
		\leq \sigma^{2t} \norm{\bfXk - \bbXk}\fs.
	\end{equation*}
	Then by virtue of the relationship \eqref{eq:cR} and the proximal smoothness of $\cM$, we can attain that
	\begin{equation} \label{eq:barXk}
	\begin{aligned}
		\sumiid \norm{\Xikn - \barXk}\fs
		\leq {} & \sumiid \norm{\proj_{\cM} \dkh{\sumjjd W^t(i, j) \Xjk + \eta \Sik} - \proj_{\cM} \dkh{\hatXk}}\fs \\
		\leq {} & 4 \sumiid \norm{\sumjjd W^t(i, j) \Xjk + \eta \Sik - \hatXk}\fs \\
		\leq {} & 8 \norm{\bfW^t \bfXk - \avXk}\fs 
		+ 8 \eta^2 \norm{\bfSk}\fs \\
		\leq {} & 8 \sigma^{2t} \norm{\bfXk - \bbXk}\fs
		+ 8 \eta^2 \norm{\bfSk}\fs.
	\end{aligned}
	\end{equation}
	Finally, from the definition of $\barXkn$, it follows that
	\begin{equation*}
		\norm{\bfXkn - \bbXkn}\fs
		= \sumiid \norm{\Xikn - \barXkn}\fs
		\leq \sumiid \norm{\Xikn - \barXk}\fs.
	\end{equation*}
	Combining the above two relationships, we complete the proof.
\end{proof}

Eventually, we conclude this subsection by bounding the tracking errors.

\begin{lemma} \label{le:error-trac}
	Suppose that Assumption \ref{asp:objective} and Assumption \ref{asp:network} hold.
	Let $\{(\bfXk, \bfDk, \bfSk)\}$ be the iterate sequence generated by Algorithm \ref{alg:DR-ProxGT} with the algorithmic parameters satisfying the condition \eqref{eq:cond-N}.
	Then for any $k \in \bN$, it holds that
	\begin{equation*}
		\norm{\bfDkn - \avDkn}\fs
		\leq 2 \sigma^{2t} \norm{\bfDk - \avDk}\fs
		+ 12 L_f^2 \norm{\bfXk - \bbXk}\fs
		+ 32 L_f^2 \eta^2 \norm{\bfSk}\fs.
	\end{equation*}
\end{lemma}

\begin{proof}
	By straightforward calculations, we can attain that
	\begin{equation*}
	\begin{aligned}
		\norm{\bfDkn - \avDkn}\fs
		= {} & \norm{\dkh{\bfW^t - \bfJ} \dkh{\bfDk - \avDk} 
		+ \dkh{I_{dn} - \bfJ} \dkh{\bfGkn - \bfGk} }\fs \\
		\leq {} & 2 \norm{\dkh{\bfW^t - \bfJ} \dkh{\bfDk - \avDk}}\fs
		+ 2 \norm{\dkh{I_{dn} - \bfJ} \dkh{\bfGkn - \bfGk}}\fs \\
		\leq {} & 2 \sigma^{2t} \norm{\bfDk - \avDk}\fs
		+ 2 \norm{\bfGkn - \bfGk}\fs.
	\end{aligned}
	\end{equation*}
	In light of the Lipschitz continuity of $\nabla f_i$, we have
	\begin{equation*}
		\norm{\bfGkn - \bfGk}\ff \leq L_f \norm{\bfXkn - \bfXk}\ff.
	\end{equation*}
	Moreover, it follows from the relationship \eqref{eq:barXk} and $\sigma^t \in (0, 1/2)$ that
	\begin{equation*}
	\begin{aligned}
		\norm{\bfXkn - \bfXk}\fs
		\leq {} & 2 \norm{\bfXkn - \bbXk}\fs
		+ 2 \norm{\bfXk - \bbXk}\fs \\
		\leq {} & 6 \norm{\bfXk - \bbXk}\fs
		+ 16 \eta^2 \norm{\bfSk}\fs.
	\end{aligned}
	\end{equation*}
	The proof is completed by collecting the above three inequalities. 
\end{proof}

%

\subsection{Sufficient Descent Property}

In this subsection, we construct a merit function to monitor the progress of Algorithm \ref{alg:DR-ProxGT}, which equips the function value with consensus and tracking errors.
While none of these three terms is guaranteed to decrease along the iterates, a suitable combination of them does satisfy a sufficient descent property.
We start from the following technical lemma.

\begin{lemma} \label{le:dist-hatX}
	Suppose that Assumption \ref{asp:network} holds.
	Let $\{(\bfXk, \bfDk, \bfSk)\}$ be the iterate sequence generated by Algorithm \ref{alg:DR-ProxGT}.
	Then for any $k \in \bN$, we have
	\begin{equation*}
	\begin{aligned}
		\norm{\hatXkn - \hatXk - \eta \hatSk}\ff
		\leq \dfrac{8 M_{pj} + \sqrt{d} M_{tg}}{d} \norm{\bfXk - \bbXk}\fs 
		+ \dfrac{2 \eta^2 M_{pj}}{d} \norm{\bfSk}\fs.
	\end{aligned}
	\end{equation*}
\end{lemma}

\begin{proof}
	To begin with, straightforward manipulations lead to that
	\begin{equation*}
	\begin{aligned}
		& \norm{\hatXkn - \hatXk - \eta \hatSk}\ff \\
		\leq {} & \dfrac{1}{d} \sumiid \norm{\Xikn - \Xik -  \proj_{\cT_{\Xik}} \dkh{\sumjjd W^t(i, j) \Xjk + \eta \Sik - \Xik}}\ff \\
		& + \dfrac{1}{d} \norm{\sumiid \proj_{\cT_{\Xik}} \dkh{\sumjjd W^t(i, j) \Xjk - \Xik}}\ff.
	\end{aligned}
	\end{equation*}
	As a direct consequence of the relationship \eqref{eq:proj-Lip-s},
	we can proceed to show that
	\begin{equation*}
	\begin{aligned}
		& \dfrac{1}{d} \sumiid \norm{\Xikn - \Xik -  \proj_{\cT_{\Xik}} \dkh{\sumjjd W^t(i, j)\Xjk + \eta \Sik - \Xik}}\ff \\
		\leq {} & \dfrac{M_{pj}}{d} \sumiid \norm{\sumjjd W^t(i, j)\Xjk + \eta \Sik - \Xik}\fs \\
		\leq {} & \dfrac{2 M_{pj}}{d} \sumiid \norm{\sumjjd W^t(i, j)\Xjk - \Xik}\fs
		+ \dfrac{2 \eta^2 M_{pj}}{d} \sumiid \norm{\Sik}\fs \\
		= {} & \dfrac{2 M_{pj}}{d} \norm{\dkh{\bfW^t - I_{dn}} \bfXk}\fs 
		+ \dfrac{2 \eta^2 M_{pj}}{d} \norm{\bfSk}\fs.
	\end{aligned}
	\end{equation*}
	Moreover, it can be readily verified that
	\begin{equation*}
		\norm{\dkh{\bfW^t - I_{dn}} \bfXk}\ff
		= \norm{\dkh{\bfW^t - I_{dn}} \dkh{\bfXk - \bbXk}}\ff
		\leq 2 \norm{\bfXk - \bbXk}\ff.
	\end{equation*}
	According to Lemma 5.3 in \cite{Deng2023decentralized},
	we have
	\begin{equation*}
		\dfrac{1}{d} \norm{\sumiid \proj_{\cT_{\Xik}} \dkh{\sumjjd W^t(i, j) \Xjk - \Xik}}\ff
		\leq \dfrac{\sqrt{d} M_{tg}}{d} \norm{\bfXk - \bbXk}\fs,
	\end{equation*}
	where $M_{tg} > 0$ is a constant.
	Collecting the above four relationships, we can obtain the assertion of this lemma.
\end{proof}

\begin{corollary} \label{cor:dist-hatXkn}
	Let $\{(\bfXk, \bfDk, \bfSk)\}$ be the iterate sequence generated by Algorithm \ref{alg:DR-ProxGT} with
	\begin{equation} \label{eq:cond-f}
		0 < \eta \leq 1,
		0 < \tau \leq \dfrac{1}{4 M_{pj} \sqrt{36 M_g^2 + L_r^2}},
		\mbox{~and~}
		t > \left\lceil \log_{\sigma}\dkh{\dfrac{1}{2}} \right\rceil.
	\end{equation}
	Suppose that Assumption \ref{asp:objective} and Assumption \ref{asp:network} hold.
	Then for any $k \in \bN$, we have
	\begin{equation*}
	\begin{aligned}
		\norm{\hatXkn - \hatXk}\fs
		\leq \dfrac{4 M_l^2 (8 M_{pj} + \sqrt{d} M_{tg})^2}{d}
		\norm{\bfXk - \bbXk}\fs
		+ \dfrac{4 \eta^2}{d}\norm{\bfSk}\fs.
	\end{aligned}
	\end{equation*}
\end{corollary}

\begin{proof}
	Combining Corollary \ref{cor:bound-S} with the condition \eqref{eq:cond-f} gives rise to that
	\begin{equation} \label{eq:bound-S}
		\eta^2 \norm{\bfSk}\fs 
		\leq \dfrac{d}{8 M_{pj}^2}.
	\end{equation}
	In light of Lemma \ref{le:dist-hatX}, we have
	\begin{equation*}
	\begin{aligned}
		\norm{\hatXkn - \hatXk}\fs
		\leq {} & 2 \norm{\hatXkn - \hatXk - \eta \hatSk}\fs 
		+ 2 \eta^2 \norm{\hatSk}\fs \\
		\leq {} & 2 \dkh{\dfrac{8 M_{pj} + \sqrt{d} M_{tg}}{d} \norm{\bfXk - \bbXk}\fs 
		+ \dfrac{2 \eta^2 M_{pj}}{d} \norm{\bfSk}\fs}^2
		+ \dfrac{2 \eta^2}{d} \norm{\bfSk}\fs \\
		\leq {} & \dfrac{4 (8 M_{pj} + \sqrt{d} M_{tg})^2}{d^2} \norm{\bfXk - \bbXk}\ff^4
		+ \dfrac{16 \eta^4 M_{pj}^2}{d^2} \norm{\bfSk}\ff^4
		+ \dfrac{2 \eta^2}{d} \norm{\bfSk}\fs \\
		\leq {} & \dfrac{4 M_l^2 (8 M_{pj} + \sqrt{d} M_{tg})^2}{d}
		\norm{\bfXk - \bbXk}\fs
		+ \dfrac{4 \eta^2}{d}\norm{\bfSk}\fs,
	\end{aligned}
	\end{equation*}
	where the last inequality is valid due to the fact that $\norm{\bfXk - \bbXk}\fs \leq d M_l^2$ and the relationship \eqref{eq:bound-S}.
	We complete the proof.
\end{proof}

The following lemma indicates that $\hatDk$ is an estimate of $\nabla f(\hatXk)$ with the approximation error controlled by the consensus error.

\begin{lemma} \label{le:hatD}
	Suppose that Assumption \ref{asp:objective} and Assumption \ref{asp:network} hold.
	Let $\{(\bfXk, \bfDk, \bfSk)\}$ be the iterate sequence generated by Algorithm \ref{alg:DR-ProxGT}.
	Then it holds that
	\begin{equation*}
		\norm{\nabla f (\hatXk) - \hatDk}\fs
		\leq \dfrac{L_f^2}{d} \norm{\bfXk - \bbXk}\fs.
	\end{equation*}
\end{lemma}

\begin{proof}
	It follows from the definition of $\hatXk$ that
	\begin{equation*}
		\sumiid \norm{\Xik - \hatXk}\fs
		\leq \sumiid \norm{\Xik - \barXk}\fs
		= \norm{\bfXk - \bbXk}\fs.
	\end{equation*}
	Since $\nabla f_i$ is Lipschitz continuous with the corresponding Lipschitz constant $L_f$, we have
	\begin{equation*}
	\begin{aligned}
		\norm{\nabla f (\hatXk) - \hatDk}\fs
		& = \norm{\dfrac{1}{d} \sumiid \dkh{\nabla f_i (\hatXk) - \nabla f_i (\Xik)}}\fs
		\leq \dfrac{1}{d} \sumiid \norm{\nabla f_i (\hatXk) - \nabla f_i (\Xik)}\fs \\
		& \leq \dfrac{L_f^2}{d} \sumiid \norm{\Xik - \hatXk}\fs
		\leq \dfrac{L_f^2}{d} \norm{\bfXk - \bbXk}\fs.
	\end{aligned}
	\end{equation*}
	The proof is completed.
\end{proof}

Now we can prove a descent inequality for the function $f$ based on the Lipschitz continuity of $\nabla f$.

\begin{proposition} \label{prop:des-f}
	Let $\{(\bfXk, \bfDk, \bfSk)\}$ be the iterate sequence generated by Algorithm \ref{alg:DR-ProxGT} with the algorithmic parameters satisfying the condition \eqref{eq:cond-f}.
	Suppose that Assumption \ref{asp:objective} and Assumption \ref{asp:network} hold.
	Then, for any $k \in \bN$, it holds that
	\begin{equation*}
		f (\hatXkn) 
		\leq f (\hatXk) 
		+ \eta \jkh{\hatDk, \hatSk}
		+ \dfrac{C_{fx}}{d} \norm{\bfXk - \bbXk}\fs
		+ \dfrac{\eta^2 C_{fs}}{d} \norm{\bfSk}\fs,
	\end{equation*}
	where $C_{fx}$ and $C_{fs}$ are two positive constants defined by
	\begin{equation*}
		C_{fx} = 8 M_g M_{pj} + \sqrt{d} M_g M_{tg} + L_f + 3 M_l^2 L_f (8 M_{pj} + \sqrt{d} M_{tg})^2,
	\end{equation*}
	and
	\begin{equation*}
		C_{fs} = 2 M_g M_{pj} + 3 L_f,
	\end{equation*}
	respectively.
\end{proposition}

\begin{proof}
	In view of the Lipschitz continuity of $\nabla f$, we have
	\begin{equation*}
	\begin{aligned}
		f (\hatXkn) 
		\leq {} & f (\hatXk) 
		+ \jkh{\nabla f (\hatXk), \hatXkn - \hatXk} 
		+ \dfrac{L_f}{2} \norm{\hatXkn - \hatXk}\fs \\
		= {} & f (\hatXk) 
		+ \jkh{\nabla f (\hatXk) - \hatDk, \hatXkn - \hatXk}
		+ \jkh{\hatDk, \hatXkn - \hatXk} \\ 
		& + \dfrac{L_f}{2} \norm{\hatXkn - \hatXk}\fs.
	\end{aligned}
	\end{equation*}
	It follows from Young's inequality that
	\begin{equation*}
		\jkh{\nabla f (\hatXk) - \hatDk, \hatXkn - \hatXk}
		\leq \dfrac{1}{L_f} \norm{\nabla f (\hatXk) - \hatDk}\fs
		+ \dfrac{L_f}{4} \norm{\hatXkn - \hatXk}\fs.
	\end{equation*}
	As a direct consequence of Lemma \ref{le:hatD}, we can proceed to show that
	\begin{equation*}
		\jkh{\nabla f (\hatXk) - \hatDk, \hatXkn - \hatXk}
		\leq \dfrac{L_f}{d} \norm{\bfXk - \bbXk}\fs
		+ \dfrac{L_f}{4} \norm{\hatXkn - \hatXk}\fs,
	\end{equation*}
	which is followed by
	\begin{equation*}
	\begin{aligned}
		f (\hatXkn) 
		\leq f (\hatXk) 
		+ \jkh{\hatDk, \hatXkn - \hatXk}
		+ \dfrac{L_f}{d} \norm{\bfXk - \bbXk}\fs
		+ \dfrac{3 L_f}{4} \norm{\hatXkn - \hatXk}\fs.
	\end{aligned}
	\end{equation*}
	Moreover, it can be readily verified that
	\begin{equation*}
	\begin{aligned}
		\jkh{\hatDk, \hatXkn - \hatXk}
		= {} & \jkh{\hatDk, \hatXkn - \hatXk - \eta \hatSk}
		+ \eta \jkh{\hatDk, \hatSk} \\
		\leq {} & \norm{\hatDk}\ff \norm{\hatXkn - \hatXk - \eta \hatSk}\ff
		+ \eta \jkh{\hatDk, \hatSk} \\
		\leq {} & \dfrac{8 M_g M_{pj} + \sqrt{d} M_g M_{tg}}{d} \norm{\bfXk - \bbXk}\fs 
		+ \dfrac{2 \eta^2 M_g M_{pj}}{d} \norm{\bfSk}\fs \\
		& + \eta \jkh{\hatDk, \hatSk},
	\end{aligned}
	\end{equation*}
	where the last inequality results from Lemma \ref{le:dist-hatX} and the relationship \eqref{eq:bound-hatD}.
	Combining the above two relationships, we can obtain that
	\begin{equation*}
	\begin{aligned}
		f (\hatXkn) 
		\leq {} & f (\hatXk) 
		+ \eta \jkh{\hatDk, \hatSk}
		+ \dfrac{8 M_g M_{pj} + \sqrt{d} M_g M_{tg} + L_f}{d} \norm{\bfXk - \bbXk}\fs \\
		& + \dfrac{2 \eta^2 M_g M_{pj}}{d} \norm{\bfSk}\fs
		+ \dfrac{3 L_f}{4} \norm{\hatXkn - \hatXk}\fs,
	\end{aligned}
	\end{equation*}
	 which together with Corollary \ref{cor:dist-hatXkn} yields that
	\begin{equation*}
		f (\hatXkn) 
		\leq f (\hatXk) 
		+ \eta \jkh{\hatDk, \hatSk}
		+ \dfrac{C_{fx}}{d} \norm{\bfXk - \bbXk}\fs
		+ \dfrac{\eta^2 C_{fs}}{d} \norm{\bfSk}\fs,
	\end{equation*}
	as desired.
\end{proof}

In order to show a similar descent inequality for the function $r$, we need the following two technical lemmas.

\begin{lemma} \label{le:hatD-proj}
	Suppose that Assumption \ref{asp:objective} holds.
	Let $\{(\bfXk, \bfDk, \bfSk)\}$ be the iterate sequence generated by Algorithm \ref{alg:DR-ProxGT}.
	Then, for any $k \in \bN$, it holds that
	\begin{equation*}
		\dfrac{1}{d} \sumiid \jkh{\Dik, \proj_{\cT_{\Xik}}(\barXk - \Xik)}
		\leq \dfrac{1}{4 d L_r} \norm{\bfDk - \avDk}\fs
		+ \dfrac{L_r + M_g M_{tg} / 2}{d} \norm{\bfXk - \bbXk}\fs.
	\end{equation*}
\end{lemma}

\begin{proof}
	To begin with, the Young's inequality leads to that
	\begin{equation*}
	\begin{aligned}
		\dfrac{1}{d} \sumiid \jkh{\Dik - \hatDk, \proj_{\cT_{\Xik}}(\barXk - \Xik)}
		\leq {} & \dfrac{1}{4 d L_r} \sumiid \norm{\Dik - \hatDk}\fs
		+ \dfrac{L_r}{d} \sumiid \norm{\Xik - \barXk}\fs \\
		= {} & \dfrac{1}{4 d L_r} \norm{\bfDk - \avDk}\fs
		+ \dfrac{L_r}{d} \norm{\bfXk - \bbXk}\fs.
	\end{aligned}
	\end{equation*}
	Since $\proj_{\cT_{\barXk}}(\cdot)$ is a linear operator, we can obtain that
	\begin{equation*}
		\dfrac{1}{d} \sumiid \jkh{\hatDk, \proj_{\cT_{\barXk}}(\barXk - \Xik)}
		= \jkh{\hatDk, \proj_{\cT_{\barXk}}(\barXk - \hatXk)} 
		= 0,
	\end{equation*}
	where the last equality follows from the fact that $\barXk - \hatXk \in \cN_{\barXk}$.
		By virtue of the Lipschitz continuity of $\proj_{\cT_{X}}(\cdot)$ with respect to $X$, we have
	\begin{equation*}
		\begin{aligned}
			\norm{\proj_{\cT_{\Xik}}(\barXk - \Xik) - \proj_{\cT_{\barXk}}(\barXk - \Xik)}\ff
			\leq \dfrac{M_{tg}}{2} \norm{\Xik - \barXk}\fs,
		\end{aligned}
	\end{equation*}
	which further implies that
	\begin{equation*}
	\begin{aligned}
		\jkh{\hatDk, \proj_{\cT_{\Xik}}(\barXk - \Xik)}
		= {} & \jkh{\hatDk, \proj_{\cT_{\Xik}}(\barXk - \Xik) - \proj_{\cT_{\barXk}}(\barXk - \Xik)} \\
		\leq {} & \norm{\hatDk}\ff \norm{ \proj_{\cT_{\Xik}}(\barXk - \Xik) - \proj_{\cT_{\barXk}}(\barXk - \Xik)}\ff \\
		\leq {} & \dfrac{M_g M_{tg}}{2} \norm{\Xik - \barXk}\fs.
	\end{aligned}
	\end{equation*}
	Finally, we can arrive at the conclusion that
	\begin{equation*}
	\begin{aligned}
		\dfrac{1}{d} \sumiid \jkh{\Dik, \proj_{\cT_{\Xik}}(\barXk - \Xik)}
		= {} &
		\dfrac{1}{d} \sumiid \jkh{\Dik - \hatDk, \proj_{\cT_{\Xik}}(\barXk - \Xik)} \\
		& + \dfrac{1}{d} \sumiid \jkh{\hatDk, \proj_{\cT_{\Xik}}(\barXk - \Xik)} \\
		\leq {} & \dfrac{1}{4 d L_r} \norm{\bfDk - \avDk}\fs
		+ \dfrac{L_r + M_g M_{tg} / 2}{d} \norm{\bfXk - \bbXk}\fs,
	\end{aligned}
	\end{equation*}
	which completes the proof.
\end{proof}

\begin{lemma} \label{le:convex-r}
	Suppose that Assumption \ref{asp:objective} holds.
	Let $\{(\bfXk, \bfDk, \bfSk)\}$ be the iterate sequence generated by Algorithm \ref{alg:DR-ProxGT} with $0 < \eta \leq 1$.
	Then, for any $k \in \bN$, it holds that
	\begin{equation*}
	\begin{aligned}
		r (\hatXkn)
		\leq {} & (1 - \eta) r(\hatXk)
		+ \dfrac{\eta}{d} \sumiid r(\Xik + \Sik)
		+ \dfrac{L_r (8 M_{pj} + \sqrt{d} M_{tg})}{d} \norm{\bfXk - \bbXk}\fs \\
		& + \dfrac{2 \eta^2 L_r M_{pj}}{d} \norm{\bfSk}\fs.
	\end{aligned}
	\end{equation*}
\end{lemma}

\begin{proof}
	It follows from the convexity of $r$ and Jensen's inequality that
	\begin{equation*}
		r(\hatXk + \hatSk)
		\leq \dfrac{1}{d} \sumiid r(\Xik + \Sik),
	\end{equation*}
	and
	\begin{equation*}
		r(\hatXk + \eta \hatSk)
		= r(\eta (\hatXk + \hatSk) + (1 - \eta) \hatXk)
		\leq \eta r(\hatXk + \hatSk) + (1 - \eta) r(\hatXk).
	\end{equation*}
	Then, in light of the Lipschitz continuity of $r$, we have
	\begin{equation*}
	\begin{aligned}
		r (\hatXkn) - r (\hatXk + \eta \hatSk)
		\leq {} & L_r \norm{\hatXkn - \hatXk - \eta \hatSk}\ff \\
		\leq {} & \dfrac{L_r (8 M_{pj} + \sqrt{d} M_{tg})}{d} \norm{\bfXk - \bbXk}\fs 
		+ \dfrac{2 \eta^2 L_r M_{pj}}{d} \norm{\bfSk}\fs,
	\end{aligned}
	\end{equation*}
	where the last inequality follows from Lemma \ref{le:dist-hatX}.
	Combining the above three inequalities, we complete the proof.
\end{proof}

Then the following proposition establishes a descent inequality for the function $r$.

\begin{proposition} \label{prop:des-r}
	Let $\{(\bfXk, \bfDk, \bfSk)\}$ be the iterate sequence generated by Algorithm \ref{alg:DR-ProxGT} with 
	\begin{equation} \label{eq:cond-r}
		0 < \eta \leq \min\{1, 2\tau\}.
	\end{equation}
	Suppose that Assumption \ref{asp:objective} and Assumption \ref{asp:network} hold.
	Then, for any $k \in \bN$, it holds that
	\begin{equation*}
	\begin{aligned}
		r (\hatXkn)
		\leq {} &  r (\hatXk)
		- \eta \jkh{\hatDk, \hatSk} 
		+ \dfrac{1}{2 d L_r} \norm{\bfDk - \avDk}\fs \\
		& + \dfrac{C_{rx}}{d} \norm{\bfXk - \bbXk}\fs
		+ \dkh{\dfrac{\eta^2 C_{rs}}{d} - \dfrac{\eta (1 - 2 L_r \tau)}{2 d \tau}} \norm{\bfSk}\fs,
	\end{aligned}
	\end{equation*}
	where $C_{rx}$ and $C_{rs}$ are two positive constants defined by
	\begin{equation*}
		C_{rx} = L_r (9 M_{pj} + \sqrt{d} M_{tg} + M_{av} + 1) + M_g M_{tg} / 2 + 1,
	\end{equation*}
	and
	\begin{equation*}
		C_{rs} = 2 L_r M_{pj},
	\end{equation*}
	respectively.
\end{proposition}

\begin{proof}
	To begin with, taking $S = \Sik$ and $S^{\prime} = \proj_{\cT_{\Xik}}(\barXk - \Xik)$ in \eqref{eq:gik} yields that
	\begin{equation*}
		g_i^{(k)} (\proj_{\cT_{\Xik}}(\barXk - \Xik)) - g_i^{(k)} (\Sik) 
		\geq \dfrac{1}{2 \tau}\norm{\proj_{\cT_{\Xik}}(\barXk - \Xik) - \Sik}\fs
		\geq 0,
	\end{equation*}
	which, after a suitable rearrangement, can be equivalently written as
	\begin{equation*}
	\begin{aligned}
		r (\Xik + \Sik) 
		\leq {} & r (\Xik + \proj_{\cT_{\Xik}}(\barXk - \Xik))
		- \jkh{\Dik, \Sik} 
		- \dfrac{1}{2 \tau} \norm{\Sik}\fs \\
		& + \jkh{\Dik, \proj_{\cT_{\Xik}}(\barXk - \Xik)} + \dfrac{1}{2 \tau} \norm{\proj_{\cT_{\Xik}}(\barXk - \Xik)}\fs.
	\end{aligned}
	\end{equation*}
	As a direct consequence of Lemma \ref{le:hatD-proj}, we can proceed to show that
	\begin{equation*}
	\begin{aligned}
		\dfrac{1}{d} \sumiid r (\Xik + \Sik)
		\leq {} & \dfrac{1}{d} \sumiid r(\Xik + \proj_{\cT_{\Xik}}(\barXk - \Xik))
		- \dfrac{1}{d} \sumiid \jkh{\Dik, \Sik} \\
		& - \dfrac{1}{2 d \tau} \norm{\bfSk}\fs
		+ \dfrac{1}{4 d L_r} \norm{\bfDk - \avDk}\fs \\
		& + \dkh{\dfrac{L_r + M_g M_{tg} / 2}{d} + \dfrac{1}{2 d \tau}} \norm{\bfXk - \bbXk}\fs.
	\end{aligned}
	\end{equation*}
	Moreover, the Lipschitz continuity of $r$ gives rise to that
	\begin{equation*}
	\begin{aligned}
		r(\Xik + \proj_{\cT_{\Xik}}(\barXk - \Xik))
		= {} & r(\Xik + \proj_{\cT_{\Xik}}(\barXk - \Xik)) - r (\barXk) \\
		& + r (\barXk) - r (\hatXk)
		+ r (\hatXk) \\
		\leq {} & L_r \norm{\barXk - \Xik - \proj_{\cT_{\Xik}}(\barXk - \Xik)}\ff \\
		& + L_r \norm{\barXk - \hatXk}\ff
		+ r (\hatXk) \\
		\leq {} & L_r M_{pj} \norm{\Xik - \barXk}\fs 
		+ \dfrac{L_r M_{av}}{d} \norm{\bfXk - \bbXk}\fs
		+ r (\hatXk),
	\end{aligned}
	\end{equation*}
	where the last inequality results from the relationships \eqref{eq:dist-aver} and \eqref{eq:proj-Lip-s}.
	Hence, we can attain that
	\begin{equation*}
	\begin{aligned}
		\dfrac{1}{d} \sumiid r (\Xik + \Sik)
		\leq {} &  r (\hatXk)
		+ \dkh{\dfrac{L_r (M_{pj} + M_{av} + 1) + M_g M_{tg} / 2}{d} + \dfrac{1}{2 d \tau}} \norm{\bfXk - \bbXk}\fs \\
		& + \dfrac{1}{4 d L_r} \norm{\bfDk - \avDk}\fs
		- \dfrac{1}{d} \sumiid \jkh{\Dik, \Sik}
		- \dfrac{1}{2 d \tau} \norm{\bfSk}\fs.
	\end{aligned}
	\end{equation*}
	In addition, it can be straightforwardly verified that
	\begin{equation*}
	\begin{aligned}
		\jkh{\hatDk, \hatSk}
		= {} & \dfrac{1}{d} \sumiid \jkh{\hatDk, \Sik}
		= \dfrac{1}{d} \sumiid \jkh{\hatDk - \Dik, \Sik}
		+ \dfrac{1}{d} \sumiid \jkh{\Dik, \Sik} \\
		\leq {} & \dfrac{1}{4 d L_r} \norm{\bfDk - \avDk}\fs 
		+ \dfrac{L_r}{d}  \norm{\bfSk}\fs 
		+ \dfrac{1}{d} \sumiid \jkh{\Dik, \Sik},
	\end{aligned}
	\end{equation*}
	which together with the condition \eqref{eq:cond-r} further results in that
	\begin{equation*}
	\begin{aligned}
		\dfrac{1}{d} \sumiid r (\Xik + \Sik)
		\leq {} &  r (\hatXk)
		+ \dfrac{L_r (M_{pj} + M_{av} + 1) + M_g M_{tg} / 2 + 1}{\eta d} \norm{\bfXk - \bbXk}\fs \\
		& + \dfrac{1}{2 \eta d L_r} \norm{\bfDk - \avDk}\fs
		- \jkh{\hatDk, \hatSk}
		+ \dkh{\dfrac{L_r}{d} - \dfrac{1}{2 d \tau}} \norm{\bfSk}\fs.
	\end{aligned}
	\end{equation*}
	The last thing to do in the proof is to combine the above relationship with Lemma \ref{le:convex-r}.
	Finally, we can obtain the assertion of this proposition.
\end{proof}

Now we are in the position to introduce the following quantity,
\begin{equation*}
	\hbar^{(k)} 
	:= f(\hatXk) + r(\hatXk) 
	+ \dfrac{\rho}{d} \norm{\bfXk - \bbXk}\fs
	+ \dfrac{\kappa}{d} \norm{\bfD - \avDk}\fs,
\end{equation*}
where $\rho > 0$ and $\kappa > 0$ are two constants defined by
\begin{equation*}
	\rho = \dfrac{2 (C_{fx} + C_{rx} + 12 \kappa L_f^2) (1 - 4 \sigma^{2t})}{1 - 8 \sigma^{2t}},
\end{equation*}
and
\begin{equation*}
	\kappa = \dfrac{1 - \sigma^{2t}}{L_r (1 - 2 \sigma^{2t})},
\end{equation*}
respectively.
We can prove that the sequence $\{\hbar^{(k)}\}$ satisfies a sufficient descent property.

\begin{corollary} \label{cor:des-h}
	Suppose that Assumption \ref{asp:objective} and Assumption \ref{asp:network} hold.
	Let $\{(\bfXk, \bfDk, \bfSk)\}$ be the iterate sequence generated by Algorithm \ref{alg:DR-ProxGT} with
	\begin{equation} \label{eq:cond-h}
	\left\{
	\begin{aligned}
		& 0 < \eta < \min\hkh{1, 2\tau, \dfrac{1 - 2 L_r \tau}{2 \tau C_s}}, \\
		& 0 < \tau <\min\hkh{\dfrac{1}{2 L_r}, \dfrac{\gamma}{24 (6\sqrt{d}M_g + L_r)}, \dfrac{1}{4 M_{pj} \sqrt{36 M_g^2 + L_r^2}}}, \\
		& t > \max\hkh{\left\lceil \log_{\sigma} \dkh{\dfrac{1}{2 \sqrt{2}}} \right\rceil, \left\lceil \log_{\sigma} \dkh{\dfrac{\gamma}{24 \sqrt{d} M_l}} \right\rceil},
	\end{aligned}
	\right.
	\end{equation}
	where $C_s = C_{fs} + C_{rs} + 8 \rho + 32 \kappa L_f^2$ is a positive constant.
 	Then for any $k \in \bN$, the following sufficient descent condition holds, which implies that the sequence $\{\hbar^{(k)}\}$ is monotonically non-increasing.
 	\begin{equation*}
 	\begin{aligned}
 		\hbar^{(k + 1)} 
 		\leq {} & \hbar^{(k)} 
 		- \dfrac{C_x (1 - 8 \sigma^{2t})}{d} \norm{\bfXk - \bbXk}\fs
 		- \dfrac{1 - 2 \sigma^{2t}}{2 d L_r} \norm{\bfDk - \avDk}\fs \\
 		& - \dkh{\dfrac{\eta (1 - 2 L_r \tau)}{2 d \tau} - \dfrac{\eta^2 C_s}{d}}\norm{\bfSk}\fs,
 	\end{aligned}
 	\end{equation*}
 	where $C_x = C_{fx} + C_{rx} + 12 \kappa L_f^2$ is a positive constant.
\end{corollary}

\begin{proof}
	Collecting Lemma \ref{le:error-cons}, Lemma \ref{le:error-trac}, Proposition \ref{prop:des-f}, Proposition \ref{prop:des-r} together, we can obtain the assertion of this corollary.
\end{proof}

\subsection{Global Convergence}

Based on the sufficient descent property of $\{\hbar^{(k)}\}$, we can finally establish the global convergence guarantee of Algorithm \ref{alg:DR-ProxGT} to a stationary point. 

\begin{theorem} \label{thm:rate}
	Let $\{(\bfXk, \bfDk, \bfSk)\}$ be the iterate sequence generated by Algorithm \ref{alg:DR-ProxGT} with the algorithmic parameters satisfying the condition \eqref{eq:cond-h}.
	Suppose that Assumption \ref{asp:objective} and Assumption \ref{asp:network} hold.
	Then $\{\bfXk\}$ has at least one accumulation point. 
	Moreover, for any accumulation point $\bfX^{\ast}$ of $\{\bfXk\}$, there exists a stationary point $X^{\ast} \in \cM$ of the problem \eqref{opt:stiefel} such that $\bfX^{\ast} = (\bfone_d \otimes I_n) X^{\ast}$. 
	Finally, the following relationships hold, which results in the global sublinear convergence rate of Algorithm \ref{alg:DR-ProxGT}.
	\begin{equation} \label{eq:sublinear}
	\left\{
	\begin{aligned}
		& \min_{k = 0, 1, \dotsc, K - 1} \norm{\bfXk - \bbXk}\fs
		\leq \dfrac{d (\hbar^{(0)} - \underline{h})}{C_x (1 - 8 \sigma^{2t}) K}, \\
		& \min_{k = 0, 1, \dotsc, K - 1} \norm{\bfDk - \avDk}\fs
		\leq \dfrac{2 d L_r (\hbar^{(0)} - \underline{h})}{(1 - 2 \sigma^{2t}) K}, \\
		& \min_{k = 0, 1, \dotsc, K - 1} \norm{\bfSk}\fs
		\leq \dfrac{2 d \tau (\hbar^{(0)} - \underline{h})}{\eta (1 - 2 \tau L_r - 2 \eta \tau C_s) K},
	\end{aligned}
	\right.
	\end{equation}
	where $\underline{h} := \inf\{f(X) + r(X) \mid X \in \conv(\cM)\}$ is a constant.
\end{theorem}

\begin{remark}
	Since $f + r$ is continuous and $\conv(\cM)$ is compact, the constant $\underline{h}$ is well-defined.
\end{remark}

\begin{proof}
	To begin with, it is obvious that $\hatXk \in \conv(\cM)$.
	Then according to the definition of the constant $\underline{h}$, we know that
	\begin{equation*}
		\hbar^{(k)} \geq f(\hatXk) + r(\hatXk) \geq \underline{h}.
	\end{equation*}
	Hence, it follows from Corollary \ref{cor:des-h} that the sequence $\{\hbar^{(k)}\}$ is convergent and the following relationships hold.		
	\begin{equation} \label{eq:limit}
		\lim\limits_{k \to \infty} \norm{\bfXk - \bbXk}\fs = 0, 
		\quad
		\lim\limits_{k \to \infty} \norm{\bfDk - \avDk}\fs = 0, 
		\quad
		\lim\limits_{k \to \infty} \norm{\bfSk}\fs = 0.
	\end{equation}
	In addition, in light of Lemma \ref{le:bound-D} and Corollary \ref{cor:bound-S}, we know that the sequence $\{(\bfDk, \bfSk)\}$ is bounded.
	And owing to the compactness of $\cM$, the sequence $\{\bfXk\}$ is also bounded.
	Then from the Bolzano-Weierstrass theorem, it follows that the sequence $\{(\bfXk, \bfDk, \bfSk)\}$ exists an accumulation point, say $(\bfX^{\ast}, \bfD^{\ast}, \bfS^{\ast})$.
	The relationships in \eqref{eq:limit} infers that $\bfX^{\ast} = (\bfone_d \otimes I_n) \barX^{\ast}$ for some $\barX^{\ast} \in \cM$, $\bfD^{\ast} = (\bfone_d \otimes I_n) \nabla f(\barX^{\ast})$, and $\bfS^{\ast} = 0$.
	
	Now we fix an arbitrary $\iid$.
	By the optimality condition of the subproblem \eqref{eq:update-S}, we have
	\begin{equation*}
		0 \in \proj_{\cT_{\Xik}} \dkh{\Dik + \dfrac{1}{\tau} \Sik + \partial r(\Xik + \Sik)},
	\end{equation*}
	which further yields that
	\begin{equation*}
		\proj_{\cT_{\Xik}} \dkh{\nabla f(\Xik + \Sik) - \Dik} - \dfrac{1}{\tau} \Sik \in \proj_{\cT_{\Xik}} \dkh{\partial h(\Xik + \Sik)}.
	\end{equation*}
	Consequently, there exists $\Eik \in \cN_{\Xik}$ such that
	\begin{equation*}
		\proj_{\cT_{\Xik}} \dkh{\nabla f(\Xik + \Sik) - \Dik} - \dfrac{1}{\tau} \Sik + \Eik \in \partial h(\Xik + \Sik).
	\end{equation*}
	Let $\{(\bfX^{(k_l)}, \bfD^{(k_l)}, \bfS^{(k_l)})\}$ be the subsequence of $\{(\bfXk, \bfDk, \bfSk)\}$ converging to $(\bfX^{\ast}, \bfD^{\ast}, \bfS^{\ast})$.
	Then we can obtain that
	\begin{equation*}
		\proj_{\cT_{X_i^{(k_l)}}} \dkh{\nabla f(X_i^{(k_l)} + S_i^{(k_l)}) - D_i^{(k_l)}} - \dfrac{1}{\tau} S_i^{(k_l)} + E_i^{(k_l)} \in \partial h(X_i^{(k_l)} + S_i^{(k_l)}).
	\end{equation*}
	Without loss of generality, we can assume that $X_i^{(k_l)} + S_i^{(k_l)} \in \cR(\beta)$ for any $l \in \bN$, where $\beta > 0$ is a constant.
	According to Proposition 2.1.2 in \cite{Clarke1990}, it follows that the set $\{H \mid H \in \partial h(X), X \in \cR(\beta)\}$ is bounded, which results in the boundedness of $E_i^{(k_l)}$.
	Without loss of generality, we can assume that the sequence $\{E_i^{(k_l)}\}$ is convergent. 
	Let $E_i^{\ast}$ denote its limiting point.
	By virtue of the relationships in \eqref{eq:limit}, we can attain that
	\begin{equation*}
		\proj_{\cT_{X_i^{(k_l)}}} \dkh{\nabla f(X_i^{(k_l)} + S_i^{(k_l)}) - D_i^{(k_l)}} - \dfrac{1}{\tau} S_i^{(k_l)} + E_i^{(k_l)} \to E_i^{\ast},
	\end{equation*}
	and
	\begin{equation*}
		X_i^{(k_l)} + S_i^{(k_l)} \to \barX^{\ast},
	\end{equation*}
	as $l \to \infty$.
	Furthermore, since $h$ is Lipschitz continuous, we also have $h(X_i^{(k_l)} + S_i^{(k_l)}) \to h(\barX^{\ast})$ as $l \to \infty$.
	Then by Remark 1(ii) in \cite{Bolte2014proximal}, it holds that
	\begin{equation*}
		E_i^{\ast} \in \partial h(\barX^{\ast}).
	\end{equation*}
	It follows from the smoothness of the projection operator $\proj_{\cN_{X}}$ with respect to $X$ that
	\begin{equation*}
		E_i^{(k_l)} = \proj_{\cN_{X_i^{(k_l)}}} \dkh{E_i^{(k_l)}} \to \proj_{\cN_{\barX^{\ast}}} \dkh{E_i^{\ast}},
	\end{equation*}
	as $l \to \infty$.
	Hence, we have $E_i^{\ast} = \proj_{\cN_{\barX^{\ast}}} \dkh{E_i^{\ast}}$ and
	\begin{equation*}
		0 = \proj_{\cT_{\barX^{\ast}}} \dkh{E_i^{\ast}} \in \proj_{\cT_{\barX^{\ast}}} \dkh{\partial h(\barX^{\ast})},
	\end{equation*}
	which indicates that $\barX^{\ast}$ is a stationary point of the problem \eqref{opt:stiefel}.
	
	Finally, we prove that the relationships in \eqref{eq:sublinear} hold.
	Indeed, it follows from Corollary \ref{cor:des-h} that
	\begin{equation*}
		\sum_{k = 0}^{K - 1} \norm{\bfXk - \bbXk}\fs 
		\leq \dfrac{d}{C_x (1 - 8 \sigma^{2t})} 
		\sum_{k = 0}^{K - 1} \dkh{\hbar^{(k)} - \hbar^{(k + 1)}}
		\leq \dfrac{d (\hbar^{(0)} - \hbar^{(K)})}{C_x (1 - 8 \sigma^{2t})}
		\leq \dfrac{d (\hbar^{(0)} - \underline{h})}{C_x (1 - 8 \sigma^{2t})},
	\end{equation*}
	which implies the first relationship in \eqref{eq:sublinear}.
	The other two relationships can be proved similarly.
	Therefore, we complete the proof.
\end{proof}

The proof previously discussed leads to the insight that the three quantities in \eqref{eq:sublinear} can be used as stationarity measures.
This observation motivates us to define the concept of $\epsilon$-stationarity for the problem \eqref{opt:stiefel}.
We say $\bfXk$ is an $\epsilon$-stationary point of \eqref{opt:stiefel} if the following condition holds:
\begin{equation*}
	\min\hkh{\norm{\bfXk - \bbXk}\ff, \norm{\bfDk - \avDk}\ff, \norm{\bfSk}\ff} \leq \epsilon.
\end{equation*}
The sublinear convergence rate established in Theorem \ref{thm:rate} indicates that the iteration complexity of Algorithm \ref{alg:DR-ProxGT} is $\cO(\epsilon^{-2})$.

\section{Numerical Experiments} \label{sec:experiments}

In this section, we conduct a comparative evaluation of the numerical performance for DR-ProxGT against state-of-the-art solvers, specifically focusing on coordinate-independent sparse estimation (CISE) \cite{Chen2010coordinate} and sparse PCA \cite{Wang2023communication} problems.
The corresponding experiments are performed on a workstation with dual Intel Xeon Gold 6242R CPU processors (at $3.10$ GHz$\times 20 \times 2$) and $510$ GB of RAM under Ubuntu 20.04.
For a fair and consistent comparison, all the algorithms under test are implemented in the $\mathtt{Python}$ language with the communication realized by the $\mathtt{mpi4py}$ package.

\subsection{Comparison on CISE Problems} \label{subsec:cise}

We first engage in a numerical comparison between DR-ProxGT and the other two benchmark methods, DRSM \cite{Wang2023decentralized} and THANOS \cite{Wang2023smoothing}, on the CISE problem \cite{Chen2010coordinate}.
The CISE model is adept at achieving sparse sufficient dimension reduction while efficiently screening out irrelevant and redundant variables, which is accomplished by solving the following optimization problem,
\begin{equation} \label{opt:cise}
	\min\limits_{X \in \Snp} \hspace{2mm} 
	-\frac{1}{2} \sumiid \tr \dkh{ X\zz A_i A_i\zz X} + \mu \norm{X}_{2, 1}.
\end{equation}
Here, $A_i \in \bR^{n \times m_i}$ represents the local data matrix privately owned by agent $i$, consisting of $m_i$ samples with $n$ features.
This model incorporates an $\ell_{2, 1}$-norm regularizer, defined as $\norm{X}_{2, 1} := \sum_{i = 1}^n \sqrt{\sum_{j = 1}^p X(i, j)^2}$, to shrink the corresponding row vectors of irrelevant variables to zero.
The parameter $\mu > 0$ modulates the row sparsity level within the model for variable selection.
For convenience, we denote $A = [A_1 \; A_2 \; \dotsb \; A_d] \in \Rnm$ as the global data matrix, where $m = m_1 + m_2 + \dotsb + m_d$.

We set the algorithmic parameters $\tau = 10$, $\eta = 1$, and $t = 1$ in DR-ProxGT.
Moreover, THANOS is equipped with the BB stepsizes proposed in \cite{Wang2022decentralized}.
For each iteration $k$, the stepsize for DRSM is set to be $0.5 / \sqrt{k}$.
In our numerical experiments, all the algorithms are started from the same initial points. 
Given the nonconvex nature of the optimization problem, different solvers may still occasionally return different solutions when starting from a common initial point at random. 
As suggested in \cite{Huang2022riemannian}, to increase the chance that all solvers find the same solution, we construct the initial point from the leading $p$ left singular vectors of $A$, which can be computed efficiently by DESTINY \cite{Wang2022decentralized} under the decentralized setting.

Our investigation encompasses a widely recognized image dataset MNIST\footnote{Available from \url{http://yann.lecun.com/exdb/mnist/}.} in the realm of machine learning research, which contains $m = 60000$ samples with $n = 784$ features.
For our testing, the samples of MNIST are uniformly distributed into $d = 16$ agents.
The CISE problem \eqref{opt:cise} is tested with $p = 10$ and $\mu = 0.001$.
The corresponding numerical experiments are performed across three different networks associated with Metropolis constant edge weight matrices \cite{Shi2015}, including the ring network, grid network, and Erd{\"o}s-R{\'e}nyi network.

For our simulation in this case, we collect and record the following two quantities at each iteration as performance metrics.
\begin{itemize}
	
	\item Stationarity violation: $\dfrac{1}{d} \norm{\bfSk}\ff$.
	
	\item Consensus error: $\dfrac{1}{d} \norm{\bfXk - \bbXk}\ff$.
	
\end{itemize}
The numerical results of this test are illustrated in Figure \ref{fig:SPCAL21}.
Each subfigure, representing a certain structure of networks, depicts the diminishing trend of stationarity violations (top panel) and consensus errors (bottom panel) against the communication rounds on a logarithmic scale.
We observe that DR-ProxGT exhibits a superior performance over DRSM and THANOS across all three networks structures.
It is worth mentioning that DR-ProxGT appears to converge asymptotically at a linear rate.
These observations underscore the effectiveness of DR-ProxGT in achieving both stationarity and consensus rapidly, highlighting its advantage in handling the CISE problem.

\begin{figure}[ht!]
	\centering
	
	\subfigure[ring network]{
		\label{subfig:SPCAL21_kkt_Ring}
		\includegraphics[width=0.3\linewidth]{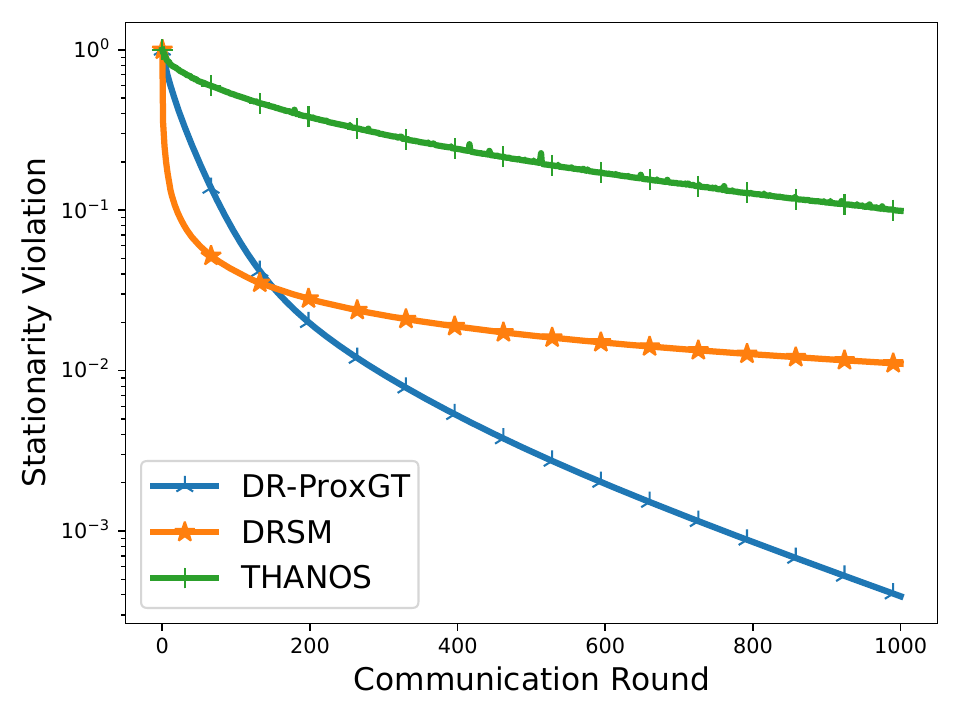}
	}
	\subfigure[grid network]{
		\label{subfig:SPCAL21_kkt_Grid}
		\includegraphics[width=0.3\linewidth]{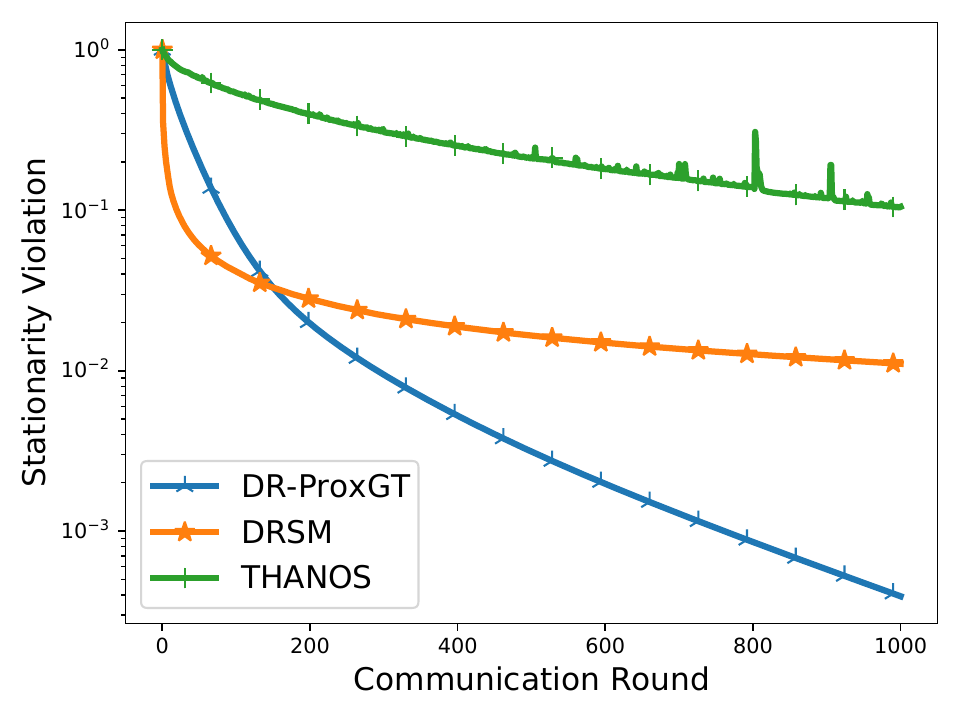}
	}
	\subfigure[Erd{\"o}s-R{\'e}nyi network]{
		\label{subfig:SPCAL21_kkt_ER}
		\includegraphics[width=0.3\linewidth]{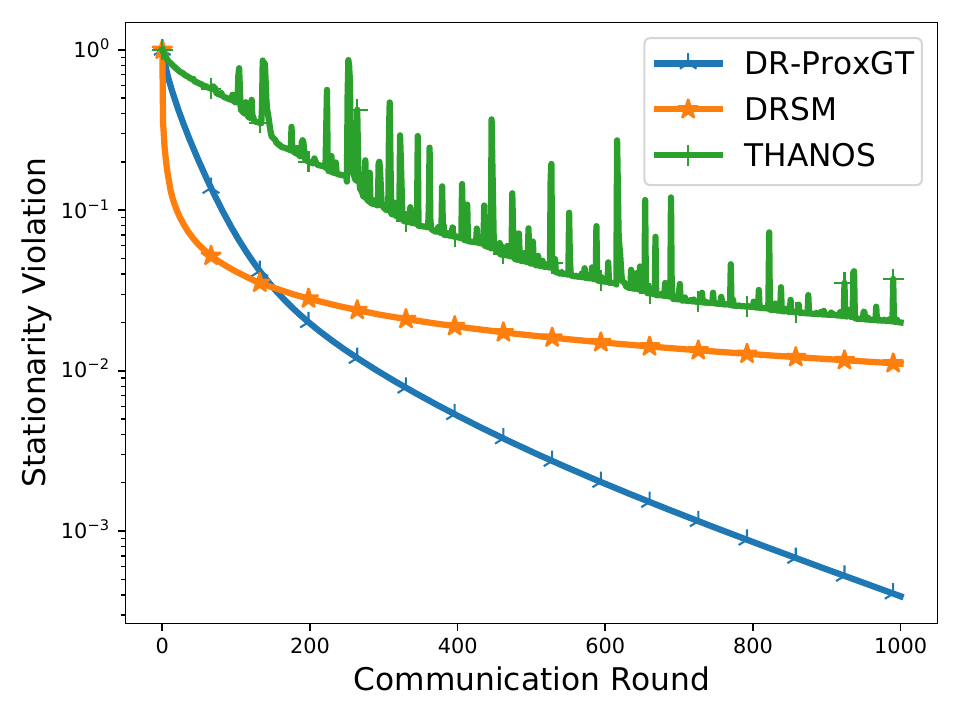}
	}
	
	\subfigure[ring network]{
		\label{subfig:SPCAL21_cons_Ring}
		\includegraphics[width=0.3\linewidth]{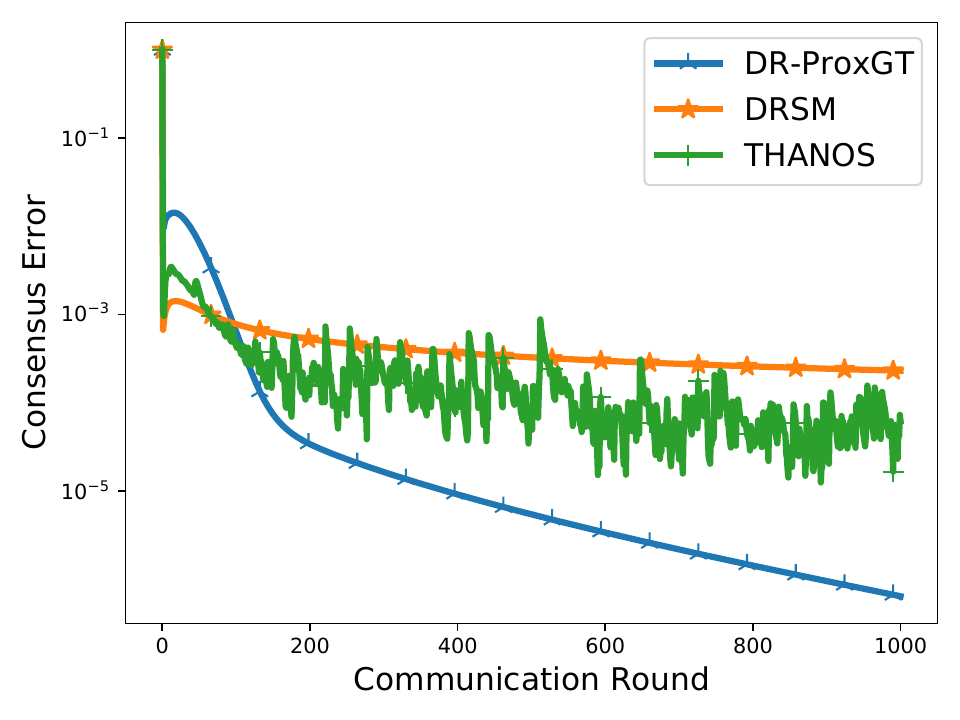}
	}
	\subfigure[grid network]{
		\label{subfig:SPCAL21_cons_Grid}
		\includegraphics[width=0.3\linewidth]{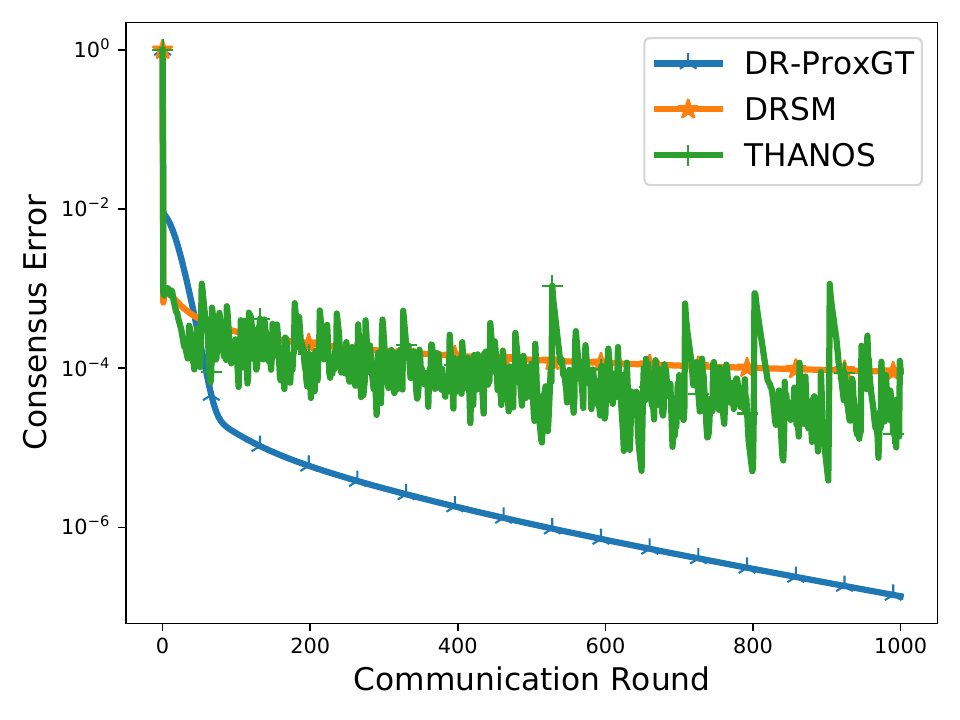}
	}
	\subfigure[Erd{\"o}s-R{\'e}nyi network]{
		\label{subfig:SPCAL21_cons_ER}
		\includegraphics[width=0.3\linewidth]{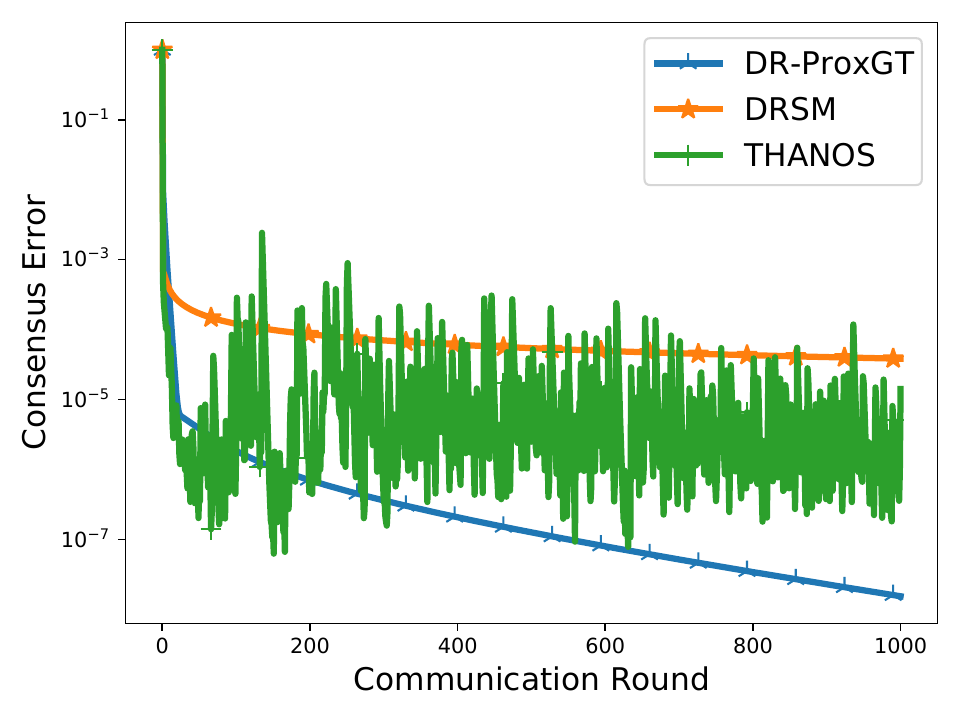}
	}
	
	\caption{Numerical comparison of DR-ProxGT, DRSM, and THANOS on different structures of networks.
		The subfigures in the first and second rows depict stationarity violations and consensus errors, respectively.}
	\label{fig:SPCAL21}
\end{figure}

\subsection{Comparison on Sparse PCA Problems} \label{subsec:spca}

The next experiment is to evaluate the numerical performance of DR-ProxGT, DRSM, and THANOS on the Global Health Estimates, which are available online\footnote{Available from \url{https://www.who.int/data/gho/data/themes/mortality-and-global-health-estimates}.}. 
The data were compiled from various sources, including national vital statistics, WHO technical programs, UN partners, the Global Burden of Disease, and scientific research. 
Our illustrative dataset contains mortality rates across $n = 19$ different age groups for $t = 20$ years (2000--2019) from $d = 50$ European countries, which is structured into a matrix $B \in \bR^{n \times dt}$.
This situation naturally lends itself to decentralized computations, as it involves data that is privately held within each country.
Such an approach is crucial for preserving the privacy of sensitive health information.

A key aspect of this study involves forecasting future mortality rates based on the Lee--Carter model \cite{Lee1992modeling,Basellini2023thirty}, where extracting the principal components of matrix $A$ is a crucial step.
In this context, we focus on solving the sparse PCA problem \cite{Wang2023communication} with the mortality dataset, aiming to enhance the interpretability and remediate the inconsistency issue.
The sparse PCA problem can be formulated as the following optimization model,
\begin{equation} \label{opt:spca}
	\min\limits_{X \in \Snp} \hspace{2mm} 
	-\frac{1}{2} \sumiid \tr \dkh{ X\zz B_i B_i\zz X} + \lambda \norm{X}_1.
\end{equation}
Here, $B_i \in \bR^{n \times t}$ contains the mortality rates collected from the $i$-th country.
The $\ell_1$-norm regularizer $\norm{X}_1 := \sum_{i = 1}^n \sum_{j = 1}^p \abs{X(i, j)}$ is imposed to promote sparsity in $X$.
The parameter $\lambda > 0$ is used to control the amount of sparseness.

In the subsequent experiment, we adopt the same stepsize strategies that are detailed in Subsection \ref{subsec:cise} for the tested algorithms.
The initial guesses are also constructed by singular value vectors.
We fix the number of principal components to be extracted at $p = 5$ and vary the parameter $\lambda$ among the values $\{0.05, 0.1, 0.15\}$. 
The three algorithms under test are configured to perform $600$ rounds of communication on an Erd{\"o}s-R{\'e}nyi network.
Numerical results from this experiment are presented in Table \ref{tb:who} with function values, sparsity levels, and consensus errors recorded.
When determining the sparsity level of a solution matrix (i.e., the percentage of zero elements), we consider a matrix element to be zero if its absolute value is less than 1e-5.
These results clearly demonstrate that DR-ProxGT outperforms DRSM and THANOS across all performance metrics by substantial margins. 
This finding is particularly noteworthy as it indicates that the superior performance of DR-ProxGT is not confined to simulated cases, but also extends to real-world applications.

\begin{table}[ht]
\caption{Numerical comparison of DR-ProxGT, DRSM, and THANOS for different values of $\lambda$.}
\label{tb:who}
\begin{tabular*}{\textwidth}{@{\extracolsep{\fill}}ccccc@{\extracolsep{\fill}}}
	\toprule%
	& Algorithm & Function Value & Sparsity Level & Consensus Error~~ \\
	\midrule
	\multirow{3}{*}{~~$\lambda = 0.05$} &
	DR-ProxGT & -7.76 & 26.32\% & 2.26e-04 \\
	& DRSM & -7.52 & 21.05\% & 2.35e-04 \\
	& THANOS  & -7.65 & 24.21\% & 4.96e-04 \\
	\midrule
	\multirow{3}{*}{~~$\lambda = 0.1$} &
	DR-ProxGT & -7.15 & 50.53\% & 3.41e-05 \\
	& DRSM & -6.98 & 30.53\% & 2.68e-04 \\
	& THANOS  & -7.02 & 31.58\% & 9.26e-05 \\
	\midrule
	\multirow{3}{*}{~~$\lambda = 0.15$} &
	DR-ProxGT & -6.54 & 64.21\% & 4.12e-05 \\
	& DRSM & -6.32 & 32.63\% & 2.51e-04 \\
	& THANOS  & -6.47 & 34.74\% & 8.25e-05 \\
	\bottomrule
\end{tabular*}
\end{table}

\section{Concluding Remarks} \label{sec:conclusions}

The composite optimization problems on Riemannian manifolds, albeit frequently encountered in many applications, present considerable challenges when addressed under the decentralized setting, primarily due to their intrinsic nonsmoothness and nonconvexity.
Current algorithms typically resort to subgradient methods or construct a smooth approximation of the objective function, which often fall short in terms of communication-efficiency.
To address this issue, we propose a novel algorithm DR-ProxGT that incorporates the gradient tracking technique into the framework of Riemannian proximal gradient methods.
We have rigorously demonstrated that DR-ProxGT achieves global convergence to a stationary point of \eqref{opt:stiefel} under mild conditions.
Furthermore, we have established the iteration complexity of $\cO(\epsilon^{-2})$, a notable improvement over existing results in the literature.
The numerical experiments illustrate that DR-ProxGT significantly surpasses existing algorithms, as evidenced by tests on CISE and sparse PCA problems.

It is worth noting that the global convergence and iteration complexity discussed in Section \ref{sec:convergence} are analyzed under the scenario of multiple consensus steps. 
Yet, from our numerical experiments, it is evident that DR-ProxGT exhibits impressive performance even with just a single consensus step per iteration. 
This observation raises a compelling question: 
\begin{center}
	{\it Can the convergence guarantee be maintained when limited to a single consensus step?}
\end{center}
This remains an intriguing area for further investigation, potentially offering more streamlined and efficient implementation possibilities for the algorithm.

\paragraph{Acknowledgement}

The work of Lei Wang and Le Bao was supported in part by the Joint United Nations Programme on HIV/AIDS and NIH/NIAID (R01AI136664 and R01AI170249).
The work of Xin Liu was supported in part by the National Natural Science Foundation of China (12125108, 12226008, 12021001, 12288201, 11991021), Key Research Program of Frontier Sciences, Chinese Academy of Sciences (ZDBS-LY-7022), and CAS AMSS-PolyU Joint Laboratory of Applied Mathematics.


\bibliographystyle{siam}

\bibliography{../../Reference/library}

\addcontentsline{toc}{section}{References}

\end{document}